\documentclass[12pt,a4paper]{amsart}
\usepackage{fancyhdr,a4wide,fontenc}
\usepackage{amsmath}
\usepackage{amssymb}
\usepackage{amsthm}
\usepackage{float}
\usepackage{caption}
\usepackage[caption = false]{subfig}
\usepackage{enumerate}
\usepackage{fullpage}
\usepackage{graphicx}
\usepackage{multicol}
\usepackage{color}

\usepackage{hyperref,lscape}

\input xy
\xyoption{all}

\usepackage[all]{xy}

\newtheorem{thm}{Theorem}

\newtheorem{lemma}[thm]{Lemma}

\theoremstyle{definition}

\newtheorem{que}[thm]{Question}
\newtheorem{ex}[thm]{Example}

\newcommand{\bfl}{{\bf{\lambda}}}
\newcommand{\bfm}{{\bf{\mu}}}

\title{
Jordan structures of nilpotent matrices in the centralizer of a nilpotent matrix with two Jordan blocks of the same size}
\author{Du\v sko Bogdani\'{c}}
\address[D. Bogdani\'{c}]{Faculty of Natural Sciences and Mathematics, University of Banja Luka, Mladena Stojanovi\' ca 2, 78000 Banja Luka, Bosnia and Herzegovina}
\email{dusko.bogdanic@gmail.com}

\author{Alen \DJ uri\'{c}}
\address[A. \DJ uri\'{c}]{Université Paris Cité, CNRS, Inria, IRIF, F-75013, Paris, France}
\email{alen.djuric@protonmail.com}

\author{Sara Koljan\v{c}i\'{c}}
\address[S. Koljan\v{c}i\'{c}]{Faculty of Natural Sciences and Mathematics, University of Banja Luka, Mladena Stojanovi\' ca 2, 78000 Banja Luka, Bosnia and Herzegovina}
\email{sara.jevdjenic@pmf.unibl.org}

\author{Polona Oblak}
\address[P.~Oblak]{Faculty of Computer and Information Science, University of Ljubljana, Ve\v cna pot 113, SI-1000 Ljubljana, Slovenia; Faculty of Mathematics and Physics, University of Ljubljana and Institute of Mathematics, Physics, and Mechanics, Jadranska ulica 19, 1000 Ljubljana, Slovenia}
\email{polona.oblak@fri.uni-lj.si}

\author{Klemen \v Sivic}
\address[K. \v Sivic]{Faculty of Mathematics and Physics, University of Ljubljana, and Institute of Mathematics, Physics, and Mechanics, Jadranska ulica 19, 1000 Ljubljana, Slovenia}
\email{klemen.sivic@fmf.uni-lj.si}

\thanks{The authors acknowledge the financial support from the bilateral project no.~BI-BA/19-20-044 funded by Slovenian Research Agency and Ministry of Civil Affairs, Bosnia and Herzegovina. Polona Oblak and Klemen \v Sivic have received funding from Slovenian Research Agency  
(research core funding no.~P1-0222 and projects no.~J1-3004 (PO\&K\v S) and N1-0103 (K\v S)) and Alen \DJ uri\'{c} has received funding from the European Union's Horizon 2020 research and innovation programme under the Marie Skłodowska-Curie grant agreement No. 754362.}
\bigskip

\date{\today}

\begin{document}

\maketitle

\begin{abstract}  
In this paper we characterize all nilpotent orbits under the action by conjugation that intersect the nilpotent centralizer of a nilpotent matrix $B$ consisting of two Jordan blocks of the same size. We list all the possible Jordan canonical forms of the nilpotent matrices that commute with $B$ by characterizing the corresponding partitions.
\end{abstract}

\noindent
%
%
\section{Introduction and background} 
Let $F$ be an arbitrary field and $N$ a positive integer. By  $M_N(F)$ we denote the set of all $N \times N$ matrices over $F$, and by $\mathcal{N}_N$ we denote the set of all nilpotent elements in $M_N(F)$. Note that the group $GL_N(F)$ acts on $\mathcal{N}_N$ by conjugation, i.e., $g\cdot X=gXg^{-1}$ for $g \in GL_N(F)$ and $X \in \mathcal{N}_N$. Recall that each nilpotent $N\times N$ matrix over $F$ is similar to a nilpotent matrix $J$ in the Jordan canonical form. If we assume that Jordan blocks of $J$ are ordered by non-increasing size, then such $J$ is uniquely defined. Hence the set of all the partitions $\bfl=(\lambda _1,\ldots ,\lambda _t)$, $\lambda _1\geq\ldots \geq\lambda _t$, $\sum_{i=1}^t \lambda_i=N$, of the integer $N$ is in a bijective correspondence with the set of all orbits of $\mathcal{N}_N$ under the action of the group $GL_N(F)$. The orbit corresponding to the partition $\bfl$ is denoted by $\mathcal{O}_{\bfl}$. It was first proved by Gerstenhaber~\cite{Gerstenhaber} 
that the dominance ordering on the set of partitions of number $N$ coincides with the inclusion ordering of orbits of $\mathcal{N}_N$. The algebraic and combinatorial properties of nilpotent orbits over various fields and in other reductive Lie algebras have been exhaustively studied, see e.g., \cite{Collingwood2017, disselhorst2021accessibility, MR909951}.

In the last decade, there has been an extensive investigation of Jordan structures of multiplication maps on finite-dimensional commutative algebras. If either the algebra is local and $x$ is an element of the maximal ideal or the algebra is graded and $x$ has a positive degree (for example, $x$ is a linear form), then the multiplication map by $x$ is a nilpotent linear map on the algebra, whose Jordan structure is called the Jordan type of the element $x$. The ``generic" Jordan type is an important invariant of the algebra, which is tightly connected to the Lefschetz properties, see e.g., \cite{Altafi2020, COSTA2019101941, LefschetzProperties, IarrobinoMarquesMcDaniel}. 

Recently, the pairs of commuting nilpotent operators have been extensively explored, see e.g., \cite{Baranovsky2001, Basili2000, Basili2003, MR1753173, Kosir1997}. Let $B\in \mathcal{N}_N$ be an arbitrary nilpotent matrix. Let $\mathcal{C}(B)$ denote the \emph{centralizer of $B$}, i.e., $\mathcal{C}(B)=\{A\in M_N(F);AB=BA\}$,
and let $$\mathcal{N}(B)=\mathcal{C}(B)\cap \mathcal{N}_N$$ denote the \emph{nilpotent centralizer of $B$}, i.e., the set of all nilpotent elements of $\mathcal{C}(B)$. It is known that $\mathcal{N}(B)$ is an irreducible variety,~\cite{Basili2003}.

One of the classical questions in the matrix theory is to determine the intersection of the nilpotent centralizer with the nilpotent orbits.

\begin{que}\label{commuting_partitions}
For a nilpotent matrix $B \in M_N(F)$, determine which of the $GL_N(F)$-orbits of $\mathcal{N}_N$ intersect the nilpotent centralizer of $B$ or, equivalently, determine all possible partitions $\bfl$ corresponding to Jordan canonical forms of the nilpotent matrices that commute with $B$.
\end{que}

In general, this seems to be a difficult question, which has been answered so far only for some specific cases of matrices $B$. 

For example, the nilpotent centralizer of a single Jordan block $J_{(N)}\in 
{\mathcal{O}_{(N)}}$ intersects only $N$ nilpotent orbits. This is the consequence of the fact that the only matrices commuting with $J_{(N)}$ are polynomials in $J_{(N)}$, see e.g., \cite[Theorem~3.2.4.2]{MR2978290}. Thus the partitions $\bfl$ that correspond to the matrices commuting with $J_{(N)}$ are precisely the ones whose parts differ by at most 1, i.e., $\lambda _1-\lambda _t\le 1$. Such partitions are called \emph{almost rectangular},~\cite{KosirOblak2009}. Following \cite{Iarrobino2018}, we denote the unique almost rectangular partition of $N$ into $k$ parts by
$$[N]^k=\left((q+1)^r,q^{k-r}\right),$$
where $q=\left\lfloor\frac{N}{k}\right\rfloor$, $r=N\mod k$, and $m^t$ abbreviates $t$ copies of $m$ in the partition. On the other hand, the matrices over a finite field or over an algebraically closed field of characteristic 0 with largest intersections with nilpotent orbits were characterized in~\cite{Britnell2011, Oblak2012}; if  $N\geq 4$, then $B^2=0$ if and only if $\mathcal{N}(B) \cap \mathcal{O}_{\bfl} \ne \emptyset$ for all partitions $\bfl$ of $N$.

For a general matrix $B$, only partial answers to Question~\ref{commuting_partitions} were given. The lower bound for the number of parts of $\bfl$ was determined in~\cite{Basili2003} and the upper bound for $\lambda_1$ was given in~\cite{Oblak2008}. Moreover, it was proved in~\cite{Baranovsky2001} that if $\bfm'$ is the conjugate partition of the Jordan canonical form of $B$, then  $\mathcal{N}(B) \cap \mathcal{O}_{\bfm'} \ne \emptyset$.

Since $\mathcal{N}(B)$ is an irreducible variety, there exists a partition $\bfl$ of $N$ such that $\mathcal{O}_{\bfl}\cap \mathcal{N}(B)$ is a dense open subset of $\mathcal{N}(B)$. If $\bfm$ is the partition of $N$ that corresponds to the Jordan canonical form of $B$, then  $\mathcal{Q}(\bfm)$ is defined to be the unique partition $\bfl$ such that $\mathcal{O}_{\bfl}\cap \mathcal{N}(B)$ is a dense open subset of $\mathcal{N}(B)$, which is the maximal partition in the dominance ordering, for which its nilpotent orbit intersects $\mathcal{N}(B)$. Recently, the conjectured recursive algorithm to compute $\mathcal{Q}(\bfm)$ (see results and discussions in~\cite{Basili2008, Basili2010, Iarrobino2013,  KosirOblak2009}) has been proved to work for general $\bfm$,~\cite{basili2022maximal}. Furthermore, the preimage $\mathcal{Q}^{-1}(\lambda)$ was completely determined for partitions $\lambda$ with two parts which differ by at least $2$  \cite{Iarrobino2018}. For generalization of $\mathcal{Q}(\mu)$ to other simple Lie algebras, see ~\cite{Panyushev2008}.

Motivated by these partial results and open problems, this paper considers and completely resolves the following special case of Question~\ref{commuting_partitions}.

\begin{que}\label{nn}
Let $B\in \mathcal{N}_{2n}$ be a nilpotent matrix with the Jordan canonical form corresponding to the partition $(n,n)$. Characterize all possible partitions $\bfl$ corresponding to Jordan canonical forms 
of the nilpotent matrices that commute with $B$.
\end{que}

There are some known partitions $\bfl$, for which the answer to Question~\ref{nn} is positive. Namely, since the centralizers of square-zero matrices intersect all nilpotent orbits, any partition of the form $\bfl=(2^t,1^s)$, $2t+s=2n$, gives rise to a nilpotent orbit that intersects ${\mathcal N}(B)$,~\cite{Britnell2011, Oblak2012}. Moreover, we can assume that $B$ is already in its Jordan canonical form, thus $B=J_{(n)}\oplus J_{(n)}$. Since every matrix commutes with its polynomial, it follows that $B$ commutes with $p(J_{(n)})\oplus q(J_{(n)})$, for any polynomials $p,q\in F[X]$. By the above results, it follows that $\bfl=([n]^t,[n]^s)$, where $1\leq t,s\leq n$,  gives a positive answer to Question~\ref{nn} as well. Also, in \cite[Theorem~3.6(c)]{Oblak2012} it was proved that any nilpotent orbit corresponding to an almost rectangular partition of $(2n)$ intersects ${\mathcal N}(B)$. Furthermore, it was proved in the same paper that all nilpotent orbits that give a positive answer to Question~\ref{nn} with the  corresponding partition $\bfl\ne (2n)$ have $\lambda_1\leq n+1$, and ${\mathcal O}_{(n+1,n-1)}$ has a non-empty intersection with $\mathcal{N}(B)$.
  It was noted in~\cite[Example~3.7]{Oblak2012} that these partitions do not constitute the complete list of nilpotent orbits intersecting ${\mathcal N}_{B}$. Some of the partitions $\bfl$ obtained by block antidiagonal and block upper triangular nilpotent matrices commuting with $B$ were listed in~\cite[Theorems~6.7, 6.19]{oblak2008orbits}.

The main objective of this paper is to completely resolve Question~\ref{nn}, and to give a complete list of nilpotent orbits $\mathcal{O}_{\bfl}$ having a non-empty intersection with ${\mathcal N}(B)$, where $B$ has the Jordan canonical form $(n,n)$. We will prove the following theorem.

\begin{thm}\label{main}
Let $B\in \mathcal{N}_{2n}$ be a nilpotent matrix with the Jordan canonical form  corresponding to the partition $(n,n)$. If $\mathcal{N}(B)$ intersects the nilpotent orbit $\mathcal{O}_{\bfl}$, then $\lambda$ is one of the partitions given in Table \ref{tab:nn}. The partitions from the cases (P1)-(P7) always belong to some nilpotent orbits intersecting $\mathcal{N}(B)$, while occurrence of the partitions from (P8)-(P9) depends on the field $F$. In particular, if $F=\mathbb{C}$, then ${\mathcal N}(B)$ intersects exactly the nilpotent orbits ${\mathcal O}_{\bfl}$, where $\bfl$ is given in Table~\ref{tab:nn}.

\begin{table}[htb]
    \centering
    \begin{tabular}{|c||p{6cm}|p{8cm}|}
        \hline
        &$\bfl$ & constraints \\
        \hline\hline
(P1)&$[2n]^{s}$ & $1\leq s \leq 2n$.\\
\hline
(P2)&$([n+m]^{m+z},[n-m]^{l-m+z})$ & $0\le z\le n-1$, $0<2m\le l< n-z$, $\lceil \frac{n+m}{m+z}\rceil \ge \lceil \frac{n-m}{l-m+z}\rceil+1$\\ 
&&or\\ 
&&$0\le z\le n-1$, $0<m<n-z$, $l=n-z$\\
\hline
(P3)&$([n+\alpha']^{m+z},[n-\alpha']^{l-m+z})$ & $0\le z\le n-1$, $0<m\le l<n-z$, $l<2m$, $n+l-m>\alpha (l+2z)$, $\alpha'=l-m+(2m-l)\alpha$\\
\hline
(P4)&$([n+\alpha'']^{l-m+z}),[n-\alpha'']^{m+z})$ & $0\le z\le n-1$, $0<m\le l<n-z$, $l<2m$, $n+m+2z\le \alpha (l+2z)<n+l+z$,\\ &&$\alpha''=m-(2m-l)\alpha$ \\
\hline 
(P5)&$([n]^{z+1},[n]^{m+z})$ & $0\le z\le n-1$, $2\le m\le n-z$\\
\hline
(P6)&$((2\alpha)^{l-\beta},(2\alpha-1)^{2z},(2\alpha-2)^{\beta})$ & $0\le z\le n-1$, $0<m\le l<n-z$, $l<2m$, $\max\{l-m-z,0\}\le \beta <\min\{m+z,l\}$\\ 
&&or\\
&&$0\le z\le n-1$, $0\le\beta<l<m\le n-z$\\
\hline
(P7)&$((2\alpha-1)^{2l-m+z-\beta},(2\alpha -2)^{2m-l+2z},(2\alpha-3)^{\beta-m-z})$ & $0\le z\le n-1$, $0< m\le l<n-z$, $l<2m$, $\max\{m+z,l\}\le \beta <\min\{2l-m+z,l+2z\}$\\
\hline
(P8)&$([n+\gamma']^{m+z-t},[n-\gamma']^{m+z+t})$ & $0\le z\le n-1$, $0<m<\frac{n-z}{2}$, $\lceil \frac{n+m+2z}{m+z}\rceil =\lceil \frac{n+m}{m+z}\rceil \ge 5 $,
$1\le t\le \min\{m,m-n+\lceil \frac{n-m}{m+z}\rceil (m+z)\}$, $\gamma'=m-\lceil \frac{n-z}{m+z}\rceil t$\\
\hline
(P9)&$([n+\gamma'']^{m+z+t},[n-\gamma'']^{m+z-t})$ &  $0\le z\le n-1$, $0<m<\frac{n-z}{2}$, $\lceil \frac{n-m}{m+z}\rceil \le \lceil \frac{n-z}{m+z}\rceil $,  $\lceil \frac{n+m}{m+z}\rceil\ge 4$, $t\ge 1$, $n+z+t\le \lceil \frac{n+m}{m+z}\rceil (m+z)\le n+2m+z-t$,\\ &&$\gamma''=m+\lceil \frac{n-m-2z}{m+z}\rceil t$\\
\hline
 \end{tabular}
    \caption{A complete list of the nilpotent orbits $\mathcal{O}_{\bfl}$ having a non-empty intersection with ${\mathcal N}(B)$, $B\in \mathcal{O}_{(n,n)}$. Here, 
    we use the notation $\alpha =\left\lceil \frac{n+z}{l+2z}\right\rceil$ and $\beta =\alpha (l+2z)-(n+z)$.}
    \label{tab:nn}
\end{table}
\end{thm}

Note that the cases (P1)-(P9) are not complementary and that some of them coincide. Since there are nine different cases, it is expected that the proof will be technical. This is why we first give some auxiliary results in Section~\ref{sec:Preliminaries}, which will be used in Section~\ref{sec:proof}, completely dedicated to proving Theorem~\ref{main}. We complete the paper with some examples in Section~\ref{sec:examples}. In particular, we show that the partitions in the cases (P8)-(P9) may not occur when $F=\mathbb{Q}$.

\section{Preliminaries}\label{sec:Preliminaries}

Let $0_n$ and $I_n$ be the zero  and the identity matrix in $M_n(F)$, respectively, where we omit the subscripts when the order of the matrix is clear from the context.

Throughout the paper, let $B\in \mathcal{N}_{2n}$ be a nilpotent matrix with the Jordan canonical form $(n,n)$. There is no loss of generality if we assume that $B$ is in the Jordan canonical form, i.e., $B=J_n\oplus J_n$ where $J_n$ denotes the $n\times n$ Jordan block. A matrix $A$ then commutes with $B$ if and only if it is of the form
$$A=\left[
\begin{array}{cccccccc}a_0&a_1&\cdots&a_{n-1}&b_0&b_1&\cdots&b_{n-1}\\&\ddots&\ddots&\vdots&&\ddots&\ddots&\vdots\\&&\ddots&a_1&&&\ddots&b_1\\&&&a_0&&&&b_0\\c_0&c_1&\cdots&c_{n-1}&d_0&d_1&\cdots&d_{n-1}\\&\ddots&\ddots&\vdots&&\ddots&\ddots&\vdots\\&&\ddots&c_1&&&\ddots&d_1\\&&&c_0&&&&d_0
\end{array}
\right]$$
for some $a_i,b_i,c_i,d_i\in F$ (see e.g., \cite{MR2228089}), and with the omitted entries all equal to zero.  Apart from this, there are two other useful descriptions of the matrices $A\in {\mathcal N}(B)$: the Weyr form~\cite{de2022algebra, MR2849857} and the polynomial notation~\cite{Neubauer1999}. In general, the notation is complicated, but in the case when two Jordan blocks of $B$ are of the same size, the notation can be simplified to the following: 
\begin{itemize}
\item(Weyr form)
Using a suitable permutation of the basis of $F^{2n}$ we can see that the matrix  $J_n\oplus J_n$ is similar to the matrix
\begin{equation}\label{eq:B-Weyr}
\left[
\begin{array}{ccccc}0_2&I_2&0_2&\cdots&0_2\\&\ddots&\ddots&\ddots&\vdots\\&&\ddots&\ddots&0_2\\&&&\ddots&I_2\\&&&&0_2
\end{array}
\right],
\end{equation}
where all blocks are of size $2\times 2$. 
Such a matrix is said to be in the Weyr canonical form and it is very useful when the commutativity questions are considered (see \cite{MR2849857}). By abuse of the notation we will denote matrix \eqref{eq:B-Weyr} again by $B$, since it represents the same linear transformation. In this basis, the matrix of the linear transformation $A$ that commutes with $B$ has the form
\begin{equation}\label{eq:block_notation}
A=\left[
\begin{array}{cccc}A_0&A_1&\cdots&A_{n-1}\\&\ddots&\ddots&\vdots\\&&\ddots&A_1\\&&&A_0
\end{array}
\right],
\end{equation}
where $A_i=\left[
\begin{array}{cc}a_i&b_i\\c_i&d_i
\end{array}
\right]$ for each $i=0,1,\ldots ,n-1$. Clearly, $A$ is nilpotent if and only if $A_0$ is nilpotent, which was already observed in~\cite[Lemma 2.3]{Basili2003}.

\item(Polynomial notation)
Denote the polynomial ring $F[X]$ by $R$ and define the following polynomials in $R$: 
$$a(X)=a_{n-1}X^{n-1}+\cdots +a_1X+a_0,\quad b(X)=b_{n-1}X^{n-1}+\cdots +b_1X+b_0,$$
$$c(X)=c_{n-1}X^{n-1}+\cdots +c_1X+c_0,\quad d(X)=d_{n-1}X^{n-1}+\cdots +d_1X+d_0.$$
Then
\begin{equation}\label{eq:poly_notation}
A=\left[
\begin{array}{cc}a(J_n)&b(J_n)\\c(J_n)&d(J_n)
\end{array}
\right].
\end{equation}
The map
$$\left[
\begin{array}{cc}a(X)&b(X)\\c(X)&d(X)
\end{array}
\right]\mapsto 
\left[
\begin{array}{cc}a(J_n)&b(J_n)\\c(J_n)&d(J_n)
\end{array}
\right]$$
is a surjective homomorphism of algebras $M_2(R)\to \mathcal{C}(B)$ with the kernel consisting of all $2\times 2$ matrices with entries in the ideal $(X^n)$, therefore, the  algebra $\mathcal{C}(B)$ is isomorphic to $M_2(F[X]/(X^n))$ (see \cite{Neubauer1999}). For simplicity, we denote the quotient ring $F[X]/(X^n)$ by $S$. Again, we will use the same notation for the matrices from $M_2(S)$ and from $\mathcal{C}(B)$. We will also use the same notation for polynomials from $R$ and for the corresponding elements of the quotient ring $S$. This should not cause any confusion, since we explicitly state which ring is considered at any definition of a new polynomial. 
We also note that in the polynomial notation the matrix  $B$ is equal to $\left[
\begin{array}{cc}X&0\\0&X
\end{array}
\right]$.
\end{itemize}

First, we state a lemma that will help us understand the powers of structured polynomial matrices.

\begin{lemma}\label{A^k}
Let $C\in M_2(R)$ be an arbitrary matrix of the form $C=\left[
\begin{array}{cc}0&1\\c(X)&d(X)
\end{array}
\right]$ for some polynomials $c(X),d(X)\in R$. Define the sequence of polynomials $s_0(X),s_1(X),\ldots \in R$ by $$s_0(X):=0, s_1(X):=1, \text{ and } s_{k+1}(X):=d(X)s_k(X)+c(X)s_{k-1}(X)$$
for $k\ge 1$. Then for each positive integer $k$, the following equalities hold:
\begin{enumerate}
\renewcommand{\theenumi}{\Alph{enumi}}
    \item $C^k=\left[ \begin{array}{cc}c(X)s_{k-1}(X)&s_k(X)\\c(X)s_k(X)&s_{k+1}(X) \end{array}
    \right],$
\item $
    s_{k-1}(X)s_{k+1}(X)-s_k(X)^2=(-1)^k c(X)^{k-1},
$
\item $s_{2k-1}(X)=\sum\limits_{i=1}^k {2k-i-1\choose i-1}c(X)^{i-1} d(X)^{2k-2i},$
\item 
$s_{2k}(X)=\sum\limits_{i=0}^{k-1}{2k-i-1\choose i}c(X)^i d(X)^{2k-2i-1}$.
\end{enumerate}
(Here, the convention $0^0=1$ is used if the corresponding polynomial is zero.)
\end{lemma}

\begin{proof}
The equality for $C^k$  in (A) can be proved by an easy induction using the recursive formula for the polynomials $s_k(X)$.
If $c(X)$ is non-zero, equality (B) follows immediately from the identity
$$(-1)^k c(X)^k=(\det C)^k=\det (C^k)=c(X)s_{k-1}(X)s_{k+1}(X)-c(X)s_k(X)^2.$$
 On the other hand, if $c(X)=0$, then $s_k(X)=d(X)^{k-1}$ for each $k\ge 1$, which completes the proof of part (B). Moreover, (C) and (D) can be shown by simultaneous induction as well.
\end{proof}

\begin{lemma}\label{A=B^zC}
 If $B\in \mathcal{N}_{2n}$ is a nilpotent matrix corresponding to the Jordan canonical form $(n,n)$ and $A \in {\mathcal N}(B)$ a non-zero matrix, then there exist a maximal non-negative integer $z$ and a matrix $C\in {\mathcal N}(B)$ such that $A=B^zC$. If $B$ is in the Weyr form and $A$ is of the form $(2)$, then $$C=\left[
\begin{array}{cccccc}
C_0&\cdots&C_{n-z-1}&0&\cdots&0\\
&\ddots&&\ddots&\ddots&\vdots\\
&&\ddots&&\ddots&0\\
&&&\ddots&&C_{n-z-1}\\
&&&&\ddots&\vdots\\
&&&&&C_0
\end{array}
\right]\in {\mathcal N}(B),$$
with $C_i$ of the form $C_i=\left[
\begin{array}{cc}
a_i&b_i\\c_i&d_i
\end{array}
\right]$ for $i \in \{0,1,\ldots ,n-z-1\}$.
\end{lemma}

\begin{proof}
Let $A \in {\mathcal N}(B)$ be as in \eqref{eq:block_notation} and let $z'$ be the smallest non-negative integer satisfying $A_{z'}\ne 0$. If $A_{z'}$ is nilpotent, we take $z=z'$ and $$C=\left[
\begin{array}{ccccccc}
A_{z'}&A_{z'+1}&\cdots&A_{n-1}&0&\cdots&0\\
&\ddots&\ddots&&\ddots&\ddots&\vdots\\
&&\ddots&\ddots&&\ddots&0\\
&&&\ddots&\ddots&&A_{n-1}\\
&&&&\ddots&\ddots&\vdots\\
&&&&&\ddots&A_{z'+1}\\
&&&&&&A_{z'}
\end{array}
\right].$$
If $A_{z'}$ is not nilpotent, then $z'\ge 1$, since $A$ is nilpotent. So we take $z=z'-1$ and $$C=\left[
\begin{array}{ccccccc}
0&A_{z'}&\cdots&A_{n-1}&0&\cdots&0\\
&\ddots&\ddots&&\ddots&\ddots&\vdots\\
&&\ddots&\ddots&&\ddots&0\\
&&&\ddots&\ddots&&A_{n-1}\\
&&&&\ddots&\ddots&\vdots\\
&&&&&\ddots&A_{z'}\\
&&&&&&0
\end{array}
\right].\hfill{\qedhere}$$ 
\end{proof}

\section{Proof of Theorem~\ref{main}}\label{sec:proof}

Let $B\in \mathcal{N}_{2n}$ be a nilpotent matrix with the Jordan canonical form $(n,n)$. By Lemma~\ref{A=B^zC}, any $A\in {\mathcal N}(B)$ is of the form  $A=B^zC$ for some non-negative integer $z$ and $C\in {\mathcal N}(B)$. If $A=0$, then the corresponding partition is $[2n]^{2n}$, which is listed in Table~\ref{tab:nn}(P1). In the rest of the proof we therefore assume that $A$ is non-zero, and hence we assume that $z$ is maximal possible such that $C$ is nilpotent. We will use the Weyr form and write $B$ as in \eqref{eq:B-Weyr} and $C,C_i$ as in Lemma~\ref{A=B^zC}, or in the polynomial notation 
and write $$ B=\left[
\begin{array}{cc}X&0\\0&X
\end{array}
\right], C=\left[
\begin{array}{cc}a(X)&b(X)\\c(X)&d(X)
\end{array}
\right], \text{ and } A= \left[
\begin{array}{cc}
X^za(X)&X^zb(X)\\
X^zc(X)&X^zd(X)
\end{array}
\right],$$
where $a(X)=a_0+\cdots +a_{n-z-1}X^{n-z-1}$, $b(X)=b_0+\cdots +b_{n-z-1}X^{n-z-1}$, $c(X)=c_0+\cdots +c_{n-z-1}X^{n-z-1},$ and $d(X)=d_0+\cdots +d_{n-z-1}X^{n-z-1}$ by Lemma~\ref{A=B^zC}. 

In the proof we consider two possibilities for $C_0$, whether $C_0$ is a non-zero or a zero matrix. 

\subsection{Case 1: \texorpdfstring{$C_0$}{TEXT} is non-zero.} 
Conjugating matrices $A$, $B$ and $C$ in the Weyr form by $G=\left[
\begin{array}{ccc}g\\&\ddots\\&&g
\end{array}
\right]\in GL_{2n}(F)$ for any $g\in GL_2(F)$  changes neither the matrix  $B$ nor the partitions corresponding to nilpotent orbits in which the matrices lie. On the other hand, each diagonal block of $G^{-1}CG$ is equal to $g^{-1}C_0g$. Since $C$ is nilpotent, $C_0$ is nilpotent as well~\cite[Lemma 2.3]{Basili2003}.  Therefore $C_0$ is similar to the nilpotent $2\times 2$ Jordan block over any field $F$, and hence we can assume that $C_0=\left[
\begin{array}{cc}0&1\\0&0
\end{array}
\right]$. In polynomial notation this means that $a(X),c(X),$ and $d(X)$ are divisible by $X$ and the constant term of $b(X)$ is 1. Choose now $G=\left[
\begin{array}{cc}b(X)&0\\-a(X)&1
\end{array}
\right]$, which is invertible in $M_2(S)$, since the constant term of $b(X)$ is non-zero. Then
$$G^{-1}AG=\left[
\begin{array}{cc}0&X^z\\X^z\left(b(X)c(X)-a(X)d(X)\right)&X^z\left(a(X)+d(X)\right)
\end{array}
\right],$$
so without any loss of generality we can assume that $a(X)=0$ and $b(X)=1$. Note that the assumption that $c(X)$ and $d(X)$ are both divisible by $X$ does not change after the above conjugation, and that after conjugation we still assume that the degrees of the polynomials $c(X)$ and $d(X)$ are bounded by $n-z-1$.

Let $l,m\le n-z$ be the largest positive integers such that $X^l$ divides $c(X)$ and $X^m$ divides $d(X)$ and let $p(X),q(X)\in R$ be unique polynomials satisfying $c(X)=X^lp(X)$ and $d(X)=X^mq(X)$. Polynomials $p(X)$ and $q(X)$ have non-zero constant terms if $l,m<n-z$, while for $l=n-z$ (respectively, $m=n-z$) the corresponding polynomial $p(X)$ (respectively, $q(X)$) is zero. 

We first consider the special case when $p(X)$ is zero.   If $p(X)=q(X)=0$, then $A^k=0$ for each $k\ge 2$ and $\mathrm{rank}\, A=n-z$, so $A$ corresponds to the partition $(2^{n-z},1^{2z})=[2n]^{n+z}$. This shows that  Table~\ref{tab:nn}(P1) is possible whenever we decompose $2n$ into an almost rectangular partition with at least $n$ parts. As noted in the introduction and proved in \cite{Britnell2011, Oblak2012}, the corresponding nilpotent orbit intersects every nilpotent centralizer.
If $p(X)=0$ and $q(X)$ is non-zero, then an easy induction shows that for each $k\ge 1$ we have
$$A^k=\left[
\begin{array}{cc}0&X^{kz+(k-1)m}q(X)^{k-1}\\0&X^{kz+km}q(X)^k
\end{array}
\right],$$
which has rank $n-kz-(k-1)m$ whenever $k\le \frac{n+m}{z+m}$, and 0 otherwise. Let $\alpha =\lceil \frac{n+m}{z+m}\rceil$ and $\beta =\alpha (z+m)-(n+m)$ (so that $0\le \beta <z+m$). Then
$$\mathrm{rank}\, A^k-\mathrm{rank}\, A^{k+1}=
\begin{cases}
n+z&k=0,\\
m+z&1\le k\le \alpha -2,\\
m+z-\beta&k=\alpha-1,\\
0&k\ge \alpha,
\end{cases}
$$
which implies that $A$ corresponds to the partition $(\alpha ^{m+z-\beta},(\alpha -1)^{\beta},1^{n-m})=([n+m]^{m+z},1^{n-m})$. These partitions belong to Table~\ref{tab:nn}(P2, $l=n-z$), and it follows from the above proof that for each of these partitions $\bfl$ there exists $A\in \mathcal{N}(B)\cap {\mathcal O}_{\bfl}$.

In the rest of Case 1 we assume that $p(X)$ is non-zero, or equivalently that $l<n-z$. We determine the partition corresponding to $A$ using the same argument as above. First we compute the powers of $A$. 
By  Lemma \ref{A^k}, we have
$$A^k=\left[
\begin{array}{cc}X^{kz}c(X)s_{k-1}(X)&X^{kz}s_k(X)\\X^{kz}c(X)s_k(X)&X^{kz}s_{k+1}(X)
\end{array}
\right]$$
for each positive integer $k$, where $s_0(X)=0$, $s_1(X)=1$, and
\begin{equation}\label{eq:s_k-recursion}
s_{k+1}(X)=d(X)s_k(X)+c(X)s_{k-1}(X)\quad \mathrm{for}\, k\ge 1.
\end{equation}
Moreover,
\begin{eqnarray}
s_{2k-1}(X)&=&\sum_{i=1}^k{2k-i-1\choose i-1}X^{2km+i(l-2m)-l}p(X)^{i-1}q(X)^{2k-2i}\quad \mathrm{and}\label{eq:s_k-odd}\\
s_{2k}(X)&=&\sum_{i=0}^{k-1}{2k-i-1\choose i}X^{2km+i(l-2m)-m}p(X)^iq(X)^{2k-2i-1}\quad \mathrm{for}\, k\ge 1\label{eq:s_k-even}.
\end{eqnarray}
Now we consider two subcases.

\subsubsection{\bf Case 1.1: $l<2m$} In this case, \eqref{eq:s_k-odd} and \eqref{eq:s_k-even} imply that the highest powers of $X$ dividing $s_{2k-1}(X)$ and $s_{2k}(X)$ are $X^{kl-l}$ and $X^{kl-l+m}$, respectively, unless $q(X)=0$ when $s_{2k}(X)=0$ for each $k\ge 1$. For each $k\ge 1$, we can write $s_{2k-1}(X)=X^{kl-l}r_{2k-1}(X)$ and $s_{2k}(X)=X^{kl-l+m}r_{2k}(X)$ for some polynomials $r_{2k-1}(X),r_{2k}(X)\in R$, where $r_{2k-1}(X)$ is uniquely determined and has non-zero constant term, while $r_{2k}(X)$ is such if $q(X)\ne 0$. Lemma~\ref{A^k} then implies that the equalities
\begin{eqnarray}
r_{2k-1}(X)r_{2k+1}(X)-X^{2m-l}r_{2k}(X)^2&=&p(X)^{2k-1},\label{eq:1.1det-odd}\\
X^{2m-l}r_{2k}(X)r_{2k+2}(X)-r_{2k+1}(X)^2&=&-p(X)^{2k}\label{eq:1.1det-even}
\end{eqnarray}
hold in $R$ for each $k\ge 1$. Furthermore, the last equality holds also for $k=0$ if we define $r_0(X)=0$.

Now we first consider the odd powers of $A$:
$$A^{2k'-1}=\left[
\begin{array}{cc}
X^{(2k'-1)z+k'l-l+m}p(X)r_{2k'-2}(X)&X^{(2k'-1)z+k'l-l}r_{2k'-1}(X)\\X^{(2k'-1)z+k'l}p(X)r_{2k'-1}(X)&X^{(2k'-1)z+k'l-l+m}r_{2k'}(X)
\end{array}
\right].$$
By definition, the polynomial $r_{2k'-1}(X)$ has non-zero constant term, so it is invertible in the quotient ring $S=R/(X^n)$. Let $r_{2k'-1}(X)^{-1}$ denote its inverse in $S$. Applying Gaussian elimination to $A^{2k'-1}$ and using \eqref{eq:1.1det-even} we now get
\begin{eqnarray*}
&&\left[
\begin{array}{cc}r_{2k'-1}(X)^{-1}&0\\-X^mr_{2k'-1}(X)^{-1}r_{2k'}(X)&1
\end{array}
\right]\cdot A^{2k'-1}\cdot \left[
\begin{array}{cc}r_{2k'-1}(X)&0\\-X^mp(X)r_{2k'-2}(X)&1
\end{array}
\right]=\\
&=&\left[
\begin{array}{cc}
0&X^{(2k'-1)z+k'l-l}\\
X^{(2k'-1)z+k'l}p(X)\Big(r_{2k'-1}(X)^2-X^{2m-l}r_{2k'-2}(X)r_{2k'}(X)\Big)&0
\end{array}
\right] = \\
&=&\left[
\begin{array}{cc}
0&X^{(2k'-1)z+k'l-l}\\
X^{(2k'-1)z+k'l}p(X)^{2k'-1}&0
\end{array}
\right].
\end{eqnarray*}
Matrices $\left[
\begin{array}{cc}r_{2k'-1}(X)^{-1}&0\\-X^mr_{2k'-1}(X)^{-1}r_{2k'}(X)&1
\end{array}
\right]$ and $ \left[
\begin{array}{cc}r_{2k'-1}(X)&0\\-X^mp(X)r_{2k'-2}(X)&1
\end{array}
\right]$ are invertible in $M_2(S)\cong \mathcal{C}(B)$, therefore,
\begin{eqnarray}
\mathrm{rank}\, A^{2k'-1}&=&\mathrm{rank} \left[
\begin{array}{cc}0&X^{(2k'-1)z+k'l-l}\\X^{(2k'-1)z+k'l}p(X)^{2k'-1}&0
\end{array}
\right]\notag\\
&=&\begin{cases}
2n-(2k'-1)(2z+l) &1\le k'\le \frac{n+z}{l+2z},\\
n-(2k'-1)z-k'l+l &\frac{n+z}{l+2z}\le k'\le \frac{n+l+z}{l+2z},\\
0 &\frac{n+l+z}{l+2z}\le k'.
\end{cases}\label{eq:1.1rank-odd}
\end{eqnarray}

For the further use we denote 
$$\alpha :=\left\lceil \frac{n+z}{l+2z}\right\rceil \text{ and } \beta :=\alpha (l+2z)-(n+z),$$ and in this case $0\le \beta < l+2z$ and $\alpha >1$. Moreover, $\frac{n+z}{l+2z}\le \frac{n+l+z}{l+2z}\le \frac{n+z}{l+2z}+1$, so
$$\left\lceil \frac{n+l+z}{l+2z}\right\rceil \in \{\alpha,\alpha+1\}.$$

Now we consider the even powers of $A$:
$$A^{2k'}=\left[
\begin{array}{cc}
X^{2k'z+k'l}p(X)r_{2k'-1}(X)&X^{2k'z+k'l-l+m}r_{2k'}(X)\\X^{2k'z+k'l+m}p(X)r_{2k'}(X)&X^{2k'z+k'l}r_{2k'+1}(X)
\end{array}
\right]$$
for $k'\ge 1$. To compute the rank of such a matrix we have to distinguish between two cases.

{\bf Case 1.1(a): $m\le l$.} In this case  $q(X)$ is 
non-zero as well. With the same argument as above we see that $r_{2k'}(X)$ is invertible in $S$ for $k'\ge 1$, and let $r_{2k'}(X)^{-1}$ be its inverse.  Using \eqref{eq:1.1det-odd}, for each $k'\ge 1$ we compute
\begin{eqnarray*}
&&\left[
\begin{array}{cc}r_{2k'}(X)^{-1}&0\\-X^{l-m}r_{2k'}(X)^{-1}r_{2k'+1}(X)&1
\end{array}
\right]\cdot A^{2k'}\cdot \left[
\begin{array}{cc}r_{2k'}(X)&0\\-X^{l-m}p(X)r_{2k'-1}(X)&1
\end{array}
\right]=\\
&=&\left[
\begin{array}{cc}
0&X^{2k'z+k'l-l+m}\\X^{2k'z+k'l+l-m}p(X)\Big(X^{2m-l}r_{2k'}(X)^2-r_{2k'-1}(X)r_{2k'+1}(X)\Big)&0
\end{array}
\right]=\\
&=&\left[
\begin{array}{cc}0&X^{2k'z+k'l-l+m}\\-X^{2k'z+k'l+l-m}p(X)^{2k'}&0
\end{array}
\right],
\end{eqnarray*}
therefore
\begin{eqnarray}
\mathrm{rank}\, A^{2k'}&=&\mathrm{rank}\, \left[
\begin{array}{cc}0&X^{2k'z+k'l-l+m}\\-X^{2k'z+k'l+l-m}p(X)^{2k'}&0
\end{array}
\right]\notag \\
&=&\begin{cases}
2n-2k'(2z+l) &k'\le \frac{n+m-l}{l+2z},\\
n-k'(2z+l)+l-m &\frac{n+m-l}{l+2z}\le k'\le \frac{n+l-m}{l+2z},\\
0 &\frac{n+l-m}{l+2z}\le k'.
\end{cases}\label{eq:1.1rank-even}
\end{eqnarray}
Moreover, it is obvious that equality \eqref{eq:1.1rank-even} holds for $k'=0$, too.

We now compare fractions in \eqref{eq:1.1rank-even} with $\alpha$. Since $\frac{n+z}{l+2z}-1\le \frac{n+m-l}{l+2z}\le \frac{n+z}{l+2z}$, we get $$\left\lceil \frac{n+m-l}{l+2z}\right\rceil\in\{\alpha -1,\alpha \}.$$
Similarly, $\frac{n+m-l}{l+2z}\le \frac{n+l-m}{l+2z}\le \frac{n+l+z}{l+2z}$, so
$$\left\lceil \frac{n+l-m}{l+2z}\right\rceil\in \{\alpha-1,\alpha , \alpha +1\}.$$
We consider various cases.

\begin{enumerate}
    \item[(i)] Assume first that $\lceil \frac{n+l-m}{l+2z}\rceil=\alpha +1$, or equivalently $\beta <l-m-z$, and hence in particular $l>m+z$. Then it is clear that $\lceil \frac{n+l+z}{l+2z}\rceil =\alpha+1$ as well. Furthermore, since $l<2m$, we have $\frac{n+l-m}{l+2z}-1\le \frac{n+m-l}{l+2z}$, so $\lceil \frac{n+m-l}{l+2z}\rceil =\alpha$. It follows from \eqref{eq:1.1rank-odd} and \eqref{eq:1.1rank-even} that
$$\mathrm{rank}\, A^{2k'}=\begin{cases}
2n-2k'(l+2z)&k'\le \alpha -1,\\
l-m-z-\beta&k'=\alpha,\\0&k'\ge \alpha+1
\end{cases}
$$
and 
$$\mathrm{rank}\, A^{2k'-1}=
\begin{cases}
2n-( 2k' -1)(l+2z )& k' \leq \alpha-1,\\
l-\beta &     k'=\alpha, \\
0,&k'\geq \alpha +1.
\end{cases}$$

In particular, $A^{2\alpha+1}=0$ and $\mathrm{rank}\, A^{2\alpha}=l-m-z-\beta$ which is strictly positive by assumption, so $2\alpha +1$ is the nilpotency index of $A$. Moreover,
$$\mathrm{rank}A^k-\mathrm{rank}\, A^{k+1}=
\begin{cases}
l+2z&k\le 2\alpha-3,\\
l+2z-\beta&k=2\alpha-2,\\
m+z&k=2\alpha-1,\\
l-m-z-\beta&k=2\alpha,\\
0&k\ge 2\alpha+1,
\end{cases}$$
which implies that $A$ corresponds to the partition Table~\ref{tab:nn}(P3)
\begin{eqnarray*}
&&((2\alpha +1)^{l-m-z-\beta},(2\alpha)^{2m-l+2z+\beta},(2\alpha-1)^{l-m+z-\beta},(2\alpha-2)^{\beta})\\
&=&([n+l-m+\alpha (2m-l)]^{m+z},[n+m-l-\alpha(2m-l)]^{l-m+z}),
\end{eqnarray*}
where $0\le z\le n-1$, $0<m\le l<n-z$, $l<2m$, $\alpha >1$ and $0\le \beta <l-m-z$, where $0<m\le l$ is redundant. We note that $n+m-l-\alpha(2m-l)$ is indeed positive for $\alpha \ge 1$, as it is equal to $\alpha (2l-2m+2z)+m-l-z-\beta \ge l-m+z-\beta>0$. 

\item[(ii)] Assume that $\lceil \frac{n+l-m}{l+2z}\rceil=\alpha$, or equivalently  $l-m-z\le \beta <2l-m+z$. We know that $\lceil \frac{n+m-l}{l+2z}\rceil \in\{\alpha-1,\alpha\}$ and $\lceil \frac{n+l+z}{l+2z}\rceil\in \{\alpha ,\alpha +1\}$. We consider all options.

\begin{itemize}

\item
Assume $\lceil \frac{n+m-l}{l+2z}\rceil =\alpha$ and $\lceil \frac{n+l+z}{l+2z}\rceil=\alpha +1$, which is equivalent to $m+z>\beta$ and $l>\beta$. In this case
$$\mathrm{rank}\, A^{2k'}=\begin{cases} 2n-2k'(l+2z)&k'\le \alpha -1,\\0&k'\ge \alpha
\end{cases}$$
and
$$\mathrm{rank}\, A^{2k'-1}=\begin{cases}
2n-(  2k' -1)(l+2z )      & k' \leq \alpha-1,\\
l-\beta &     k'=\alpha, \\
0&k'\ge \alpha+1.
\end{cases}$$
We therefore get
$$\mathrm{rank}\, A^k-\mathrm{rank}\, A^{k+1}=\begin{cases}
l+2z&k\le 2\alpha-3,\\l+2z-\beta&k=2\alpha-2,\\l-\beta&k=2\alpha -1,\\0&k\ge 2\alpha,
\end{cases}$$
so $A \in {\mathcal O}_{\bfl}$, where
$$\bfl = ((2\alpha)^{l-\beta},(2\alpha-1)^{2z},(2\alpha-2)^{\beta}),$$
 corresponds to the partition Table~\ref{tab:nn}(P6), and $0\le z\le n-1$, $0<m\le l<n-z$, $l<2m$ and $\max\{l-m-z,0\}\le \beta <\min\{m+z,l\}$. (Again, the condition $\beta <2l-m+z$ is redundant.) 

\item
Assume $\lceil \frac{n+m-l}{l+2z}\rceil =\lceil \frac{n+l+z}{l+2z}\rceil=\alpha$, which is equivalent to $l\le \beta <m+z$. Then
$$\mathrm{rank}\, A^{2k'}=
\begin{cases}
2n-2k'(l+2z)&k'\le \alpha-1,\\
0&k'\ge \alpha,
\end{cases}$$
$$\mathrm{rank}\, A^{2k'-1}=\begin{cases}
2n-(  2k' -1)(l+2z )      & k' \leq \alpha-1,\\
0&k'\geq \alpha,
\end{cases}$$
and 
$$\mathrm{rank}\, A^k-\mathrm{rank}\, A^{k+1}=\begin{cases}
l+2z&k\le 2\alpha-3,\\  2l+2z-2\beta&k=2\alpha-2,\\0&k\ge 2\alpha-1.
\end{cases}$$
Consequently, $A$ corresponds to the partition Table~\ref{tab:nn}(P1)
$$\left( (2\alpha-1)^{2l+2z-2\beta},(2\alpha-2)^{2\beta-l}\right)=[2n]^{l+2z},$$
where $0\le z\le n-1$, $0< m\le l<n-z$, $l<2m$ and $l\le \beta <m+z$ (while the conditions $l-m-z\le \beta <2l-m+z$ and $\beta<l+2z$ are redundant).

\item

Assume $\lceil \frac{n+m-l}{l+2z}\rceil =\alpha -1$ and $\lceil \frac{n+l+z}{l+2z}\rceil =\alpha +1$, which is equivalent to $m+z\le \beta <l$. It follows that 
$$\mathrm{rank}\, A^{2k'}=\begin{cases}
2n-2k'(l+2z)&k'\le \alpha -2,\\ 
2l-m+z-\beta	&k'=\alpha -1,\\0&k'\ge \alpha,
\end{cases}$$
$$\mathrm{rank}\, A^{2k'-1}=\begin{cases}
2n-(2k'-1)(l+2z)&k'\le \alpha -1,\\ 
l-\beta&k'=\alpha,\\
0&k'\ge \alpha+1
\end{cases}$$
and
$$\mathrm{rank}\, A^k-\mathrm{rank}\, A^{k+1}=\begin{cases}
l+2z&k\le 2\alpha-4,\\
l+m+3z-\beta&k=2\alpha-3,\\
l-m+z&k=2\alpha-2,\\
l-\beta&k=2\alpha-1,\\
0&k\ge 2\alpha.
\end{cases}$$
Consequently, $A$ corresponds to the partition Table~\ref{tab:nn}(P4)
\begin{eqnarray*}
&&((2\alpha)^{l-\beta},(2\alpha-1)^{z-m+\beta},(2\alpha-2)^{2m+2z-\beta},(2\alpha-3)^{\beta-m-z})\\
&=&([n+m-\alpha (2m-l)]^{l-m+z},[n-m+\alpha(2m-l)]^{m+z}),
\end{eqnarray*}
where $0\le z\le n-1$, $0<m\le l<n-z$, $l<2m$ and $m+z\le \beta <l$ (while the conditions $l-m-z\le \beta <2l-m+z$ and also $m\le l$ are redundant).

\item

It remains to consider the case when $\lceil \frac{n+m-l}{l+2z}\rceil =\alpha -1$ and $\lceil \frac{n+l+z}{l+2z}\rceil =\alpha$, which is equivalent to $\beta \ge \max\{m+z,l\}$. 
Now we have
$$\mathrm{rank}\, A^{2k'}=\begin{cases}
2n-2k'(l+2z)&k'\le \alpha -2,\\  
2l-m+z-\beta	&k'=\alpha -1,\\
0&k'\ge \alpha,
\end{cases}$$
$$\mathrm{rank}\, A^{2k'-1}=\begin{cases}
2n-(  2k' -1)(l+2z )      & k' \leq \alpha-1,\\
0&k'\geq \alpha,
\end{cases}$$
and 
$$\mathrm{rank}\, A^k-\mathrm{rank}\, A^{k+1}=\begin{cases}
l+2z&k\le 2\alpha-4,\\  
l+m+3z-\beta&k=2\alpha-3,
\\2l-m+z-\beta&k=2\alpha-2,
\\0&k\ge 2\alpha -1.
\end{cases}$$

Hence, $A$ corresponds to the partition Table~\ref{tab:nn}(P7)
$$\left( (2\alpha-1)^{2l-m+z-\beta},(2\alpha -2)^{2m-l+2z},(2\alpha-3)^{\beta-m-z}\right),$$
where $0\le z\le n-1$, $0<m\le l<n-z$, $l<2m$ and  $\max\{m+z,l\}\le \beta <\min\{2l-m+z,l+2z\}$ (and the condition $\beta \ge l-m-z$ is redundant).

\end{itemize}

\item[(iii)] Assume now that $\lceil \frac{n+l-m}{l+2z}\rceil=\alpha -1$, or equivalently  $\beta \ge 2l-m+z$, so in particular we have $z>l-m$, since $\beta <l+2z$. 
Since $l\ge m$, we get also $\lceil \frac{n+m-l}{l+2z}\rceil=\alpha-1$. 
Moreover, the conditions $m\le l$ and $z\ge 0$ imply $2l-m+z\ge l$, so $\beta \ge l$, which is equivalent to 
$\lceil \frac{n+l+z}{l+2z}\rceil =\alpha$. 
We thus get
$$\mathrm{rank}\, A^{2k'}=\begin{cases}
2n-2k'(l+2z)&k'\le \alpha -2,\\
0&k'\ge \alpha-1,
\end{cases}
$$
$$\mathrm{rank}\, A^{2k'-1}=
\begin{cases}
2n-(2k'-1)(l+2z)&k'\le \alpha -1,\\
0&k'\ge \alpha,
\end{cases}$$
and therefore
$$\mathrm{rank}\, A^k-\mathrm{rank}\, A^{k+1}=
\begin{cases}
l+2z&k\le 2\alpha -4,\\
2n-(2\alpha -3)(l+2z)&k=2\alpha -3,\\
0&k\ge2\alpha -2.
\end{cases}
$$
It follows that $A \in {\mathcal O}_{\bfl}$, where
$$\bfl=((2\alpha -2)^{2n-(2\alpha -3)(l+2z)},(2\alpha -3)^{(2\alpha -2)(l+2z)-2n})=[2n]^{l+2z}$$ 
as in Table~\ref{tab:nn}(P1), for $0\le z\le n-1$, $0<m\le l<n-z$, $l<2m$ and $2l-m+z\le \beta <l+2z$.
\end{enumerate}

At the end of Case 1.1(a) we observe that it immediately follows from the above proof that for each of the partitions $\bfl$ obtained in (i)-(iii) there exists $A\in \mathcal{N}(B)\cap {\mathcal O}_{\bfl}$.

Note that in the special case when $l=1$, the condition (a) forces $m=1$. Moreover, the option (i) above and the last two cases in (ii) are not possible, and in all other cases we get the partition $[2n]^{2z+1}$.  Hence, in particular we get all odd cases in the first set of partitions listed in the theorem, except $[2n]^{2n-1}$, since $1<n-z$. However, the existence of this partition was shown at the beginning of the proof in the case when $p(X)=q(X)=0$. The existence of all even cases will be proved in Case 2 below, but we note here that we get almost all of them also from the partitions above in the case $l=2$. Indeed, in this case the condition $m\le l<2m$ implies $m=2$, again the option (i) and the last two cases in (ii) are not possible, and the only case when we get a partition different from $[2n]^{2z+2}$ is when $\beta=1$. Note also that for $n\ge 2$ the partition $[2n]^{2n-2}$ is not covered by the above cases if $l=2$, but it has been established when we considered the case $p(X)=q(X)=0$.

{\bf Case 1.1(b): $n-z\ge m>l$}. We proceed as in Case 1.1(a). As above, the polynomial $r_{2k'+1}(X)$ has non-zero constant term, therefore it is invertible in $S=R/(X^n)$ for each $k'\ge 0$. Let $r_{2k'+1}(X)^{-1}$ be its inverse in $S$. For $k'\ge 1$ we compute, using \eqref{eq:1.1det-odd}, that
\begin{eqnarray*}
&&\left[
\begin{array}{cc}1&-X^{m-l}r_{2k'+1}(X)^{-1}r_{2k'}(X)\\0&r_{2k'+1}(X)^{-1}
\end{array}
\right]\cdot A^{2k'}\cdot \left[
\begin{array}{cc}r_{2k'+1}(X)&0\\X^mp(X)r_{2k'}(X)&1
\end{array}
\right]=\\
&=&\left[
\begin{array}{cc}
X^{2k'z+k'l}p(X)\Big(r_{2k'-1}(X)r_{2k'+1}(X)-X^{2m-l}r_{2k'}(X)^2\Big)&0\\0&X^{2k'z+k'l}
\end{array}
\right]=\\
&=&\left[
\begin{array}{cc}X^{2k'z+k'l}p(X)^{2k'}&0\\0&X^{2k'z+k'l}
\end{array}
\right],
\end{eqnarray*}
so
\begin{eqnarray}
\mathrm{rank}\, A^{2k'}&=&\mathrm{rank}\,\left[
\begin{array}{cc}X^{2k'z+k'l}p(X)^{2k'}&0\\0&X^{2k'z+k'l}
\end{array}
\right]=\notag \\
&=&\left\{
\begin{array}{ll}
2n-2k'(l+2z)&k'\le \frac{n}{l+2z},\\
0&k'\ge \frac{n}{l+2z},
\end{array}
\right.\label{eq:1.1(b)rank-even}
\end{eqnarray}
and clearly $\mathrm{rank}\, A^0=2n$, so \eqref{eq:1.1(b)rank-even} holds for $k'=0$. Since $\frac{n+z}{l+2z}-1\le \frac{n}{l+2z}\le \frac{n+z}{l+2z}$, we get
$$\left\lceil \frac{n}{l+2z}\right\rceil\in \{\alpha -1,\alpha\}.$$
On the other hand, we know that $\lceil \frac{n+l+z}{l+2z}\rceil \in \{\alpha ,\alpha +1\}$. Moreover, $\frac{n+l+z}{l+2z}-\frac{n}{l+2z}\le 1$, therefore $\lceil \frac{n+l+z}{l+2z}\rceil -\lceil \frac{n}{l+2z}\rceil \in \{0,1\}$.  We consider all options.

\begin{enumerate}
\item[(i)]
Assume first that $\lceil \frac{n+l+z}{l+2z}\rceil =\alpha +1$, which is equivalent to $\beta <l$. Then it is clear that $\lceil \frac{n}{l+2z}\rceil =\alpha$. Furthermore, 
we get from \eqref{eq:1.1rank-odd} and \eqref{eq:1.1(b)rank-even} that
$$\mathrm{rank}\, A^{2k'}=\left\{
\begin{array}{ll}2n-2k'(l+2z)&k'\le \alpha-1,\\0&k'\ge \alpha,
\end{array}
\right.
$$
$$\mathrm{rank}\, A^{2k'-1}=\left\{
\begin{array}{ll}2n-(2k'-1)(l+2z)&k'\le \alpha-1,\\
l-\beta&k'=\alpha,\\
0&k'\ge \alpha+1,
\end{array}\right.
$$
and
$$\mathrm{rank}\, A^k-\mathrm{rank}\, A^{k+1}=
\left\{
\begin{array}{ll}l+2z&k\le 2\alpha-3,\\
l+2z-\beta&k=2\alpha-2,\\
l-\beta&k=2\alpha-1,\\
0&k\ge 2\alpha,
\end{array}
\right.
$$
and the corresponding partition is Table~\ref{tab:nn}(P6)
$$\left( (2\alpha)^{l-\beta},(2\alpha-1)^{2z},(2\alpha-2)^{\beta}\right),$$ for $0\le z\le n-1$, $0< l<m\le n-z$ and $0\le \beta <l$. 

\item[(ii)]
Assume now that $\lceil \frac{n+l+z}{l+2z}\rceil =\lceil \frac{n}{l+2z}\rceil =\alpha$, which is equivalent to $l\le \beta <l+z$. 
It follows from \eqref{eq:1.1rank-odd} and \eqref{eq:1.1(b)rank-even} that
$$\mathrm{rank}\, A^{2k'}=\left\{
\begin{array}{ll}2n-2k'(l+2z)&k'\le \alpha-1,\\0&k'\ge \alpha,
\end{array}
\right.
$$
$$\mathrm{rank}\, A^{2k'-1}=\left\{
\begin{array}{ll}2n-(2k'-1)(l+2z)&k'\le \alpha-1,\\0&k'\ge \alpha,
\end{array}\right.
$$
and 
$$\mathrm{rank}\, A^k-\mathrm{rank}\, A^{k+1}=
\left\{
\begin{array}{ll}l+2z&k\le 2\alpha-3,\\2n-2(\alpha -1)(l+2z)&k=2\alpha-2,\\0&k\ge 2\alpha-1,
\end{array}
\right.
$$
so the corresponding partition is  Table~\ref{tab:nn}(P1) $$\left( (2\alpha-1)^{2n-2(\alpha-1)(l+2z)}, (2\alpha-2)^{-2n+(2\alpha -1)(l+2z)}\right)=[2n]^{l+2z},$$ where $0\le z\le n-1$, $0<l<m\le n-z$ and $l\le \beta <l+z$.

\item[(iii)]
It remains to consider the case when  $\lceil \frac{n}{l+2z}\rceil =\alpha -1$ (and hence $\lceil \frac{n+l+z}{l+2z}\rceil =\alpha$), which is equivalent to $\beta \ge l+z$. In this case 
 we have 
$$\mathrm{rank}\, A^{2k'}=\left\{
\begin{array}{ll}2n-2k'(l+2z)&k'\le \alpha-2,\\0&k'\ge \alpha-1,
\end{array}
\right.
$$
$$\mathrm{rank}\, A^{2k'-1}=\left\{
\begin{array}{ll}2n-(2k'-1)(l+2z)&k'\le \alpha-1,\\0&k'\ge \alpha,
\end{array}\right.
$$
and 
$$\mathrm{rank}\, A^k-\mathrm{rank}\, A^{k+1}=
\left\{
\begin{array}{ll}l+2z&k\le 2\alpha-4,\\2n-(2\alpha -3)(l+2z)&k=2\alpha-3,\\0&k\ge 2\alpha-2,
\end{array}
\right. 
$$
and the corresponding partition is Table~\ref{tab:nn}(P1)
$$\left( (2\alpha-2)^{2n-(2\alpha -3)(l+2z)}, (2\alpha-3)^{-2n+(2\alpha -2)(l+2z)}\right)=[2n]^{l+2z},$$ where $0\le z\le n-1$, $0<l<m\le n-z$ and $\beta \ge l+z$.
\end{enumerate}

As before, it follows from the above proof that for each of the partitions obtained in (i)-(iii) there indeed exists $A\in \mathcal{N}(B)$ that corresponds to the partition.

\subsubsection{\bf Case 1.2:  $l\ge 2m$} Then $q(X)$ has non-zero constant term. Moreover, it follows from \eqref{eq:s_k-odd} and \eqref{eq:s_k-even} that, for each $k\ge 1$, the polynomial $s_k(X)$ is divisible by $X^{km-m}$. For each $k\ge 1$, we write $s_k(X)=X^{km-m}r_k(X)$ for some uniquely determined polynomial $r_k(X)\in R$, and let $r_0(X)=0$. Then
\begin{equation}\label{eq:r_k-recursion}
r_{k+1}(X)=q(X)r_k(X)+X^{l-2m}p(X)r_{k-1}(X)
\end{equation}
for $k\ge 1$ by \eqref{eq:s_k-recursion}. Furthermore, Lemma~\ref{A^k} implies that
$$
A^k=\left[
\begin{array}{cc}
X^{kz+km-2m+l}p(X)r_{k-1}(X)&X^{kz+km-m}r_k(X)\\X^{kz+km-m+l}p(X)r_k(X)&X^{kz+km}r_{k+1}(X)
\end{array}
\right]
$$
for $k\ge 1$, and
\begin{equation}\label{eq:1.2det}
r_{k-1}(X)r_{k+1}(X)-r_k(X)^2=(-1)^kX^{(k-1)(l-2m)}p(X)^{k-1}
\end{equation}
holds in $R$ for each $k\ge 2$. Clearly, \eqref{eq:1.2det} holds also for $k=1$. If $l>2m$, then it is also clear from \eqref{eq:s_k-odd} and \eqref{eq:s_k-even} that $r_k(X)$ has non-zero constant term for each $k\ge 1$.

 As we will see below, it is convenient to define the numbers
$$\alpha =\left\lceil\frac{n+m}{z+m}\right\rceil, \beta =\alpha (z+m)-(n+m), \gamma =\left\lceil \frac{n-m}{z+l-m}\right\rceil,  \mathrm{ and }\, \delta =\gamma(z+l-m)-(n-m).$$
Then $0\le \beta <z+m$ and $0\le \delta < z+l-m$. Moreover, since $l\ge 2m$, we get $z+l-m\ge z+m$, hence $\frac{n+m}{z+m}\ge \frac{n-m}{z+l-m}$, and consequently $\alpha\geq \gamma.$ Furthermore, $l<n-z$ implies $\gamma >1$.

Suppose first that $k\ge 1$ is such that $r_k(X)$ is not divisible by $X$. (As noted above, this happens always if $l>2m$.) Then it is invertible in $S=R/(X^n)$, and let $r_k(X)^{-1}$ be its inverse in $S$. Using \eqref{eq:1.2det} we now compute
\begin{eqnarray*}
&&\left[
\begin{array}{cc}
r_k(X)^{-1}&0\\-X^mr_k(X)^{-1}r_{k+1}(X)&1
\end{array}
\right]\cdot A^k\cdot \left[
\begin{array}{cc}r_k(X)&0\\-X^{l-m}p(X)r_{k-1}(X)&1
\end{array}
\right]=\\
&=&\left[
\begin{array}{cc}
0&X^{kz+km-m}\\X^{kz+km+l-m}p(X)\Big(r_k(X)^2-r_{k-1}(X)r_{k+1}(X)\Big)&0
\end{array}
\right]=\\
&=&\left[
\begin{array}{cc}
0&X^{kz+km-m}\\(-1)^{k+1}X^{kz+kl-km+m}p(X)^k&0
\end{array}
\right].
\end{eqnarray*}
We can therefore compute
\begin{eqnarray*}\label{eq:14}
\mathrm{rank}\, A^{k}&=&\mathrm{rank}\,
\left[
\begin{array}{cc} 0 & X^{kz+km-m}\\ (-1)^{k+1}X^{kz+kl-km+m}p(X)^{k} &    0
\end{array} \right]\notag \\
&=&
\begin{cases}
2n-k(l+2z)& k\leq \frac{n-m}{z+l-m},\\
n+m-k(m+z)&   \frac{n-m}{z+l-m} \leq k\le \frac{n+m}{z+m}, \\
0&  k\ge  \frac{n+m}{z+m}.
\end{cases}
\end{eqnarray*}
If $k\ge 1$ and $r_k(X)$ is not divisible by $X$, 
we therefore get
\begin{equation}\label{eq:1.2rankA^k-not_divisible}
\mathrm{rank}\, A^k=\left\{
\begin{array}{ll}2n-k(l+2z)&k\le \gamma-1,\\
n+m-k(m+z)&\gamma\le k\le \alpha -1,\\
0&k\ge \alpha,
\end{array}
\right.
\end{equation}
and it is clear that the above holds also for $k=0$.

Suppose now that $X$ divides $r_k(X)$ for some $k\ge 1$. We already know that this forces $l=2m$, hence $\gamma =\lceil \frac{n-m}{z+m}\rceil$, and in particular $\alpha \le \gamma +2$. Moreover, if $X$ divides $r_{k-1}(X)$, too, then \eqref{eq:r_k-recursion} implies that $r_j(X)$ is divisible by $X$ for any $j$. (Recall that $p(X)$ is not divisible by $X$.) However, this contradicts $r_1(X)=1$. It follows that no two consecutive polynomials $r_j(X)$ and $r_{j-1}(X)$ are divisible by $X$.

Let $t$ be the largest positive integer smaller than or equal to $m+1$ such that $X^t$ divides $r_k(X)$, and let $\widetilde{r_k}(X)\in R$ be the unique polynomial satisfying $r_k(X)=X^t\widetilde{r_k}(X)$. If $t\le m$, then the polynomial $\widetilde{r_k}(X)$ has non-zero constant term, so it is invertible in $S=R/(X^n)$. Let $\widetilde{r_k}(X)^{-1}$ be its inverse in $S$. For $t\le m$, we get by \eqref{eq:1.2det} that
\begin{eqnarray*}
&&\left[
\begin{array}{cc}\widetilde{r_k}(X)^{-1}&0\\-X^{m-t}\widetilde{r_k}(X)^{-1}r_{k+1}(X)&1
\end{array}
\right]
\cdot A^k\cdot \left[
\begin{array}{cc}\widetilde{r_k}(X)&0\\-X^{m-t}p(X)r_{k-1}(X)&1
\end{array}
\right]=\\
&=&\left[
\begin{array}{cc}0&X^{kz+km-m+t}\\X^{kz+km+m-t}p(X)\Big(r_k(X)^2-r_{k-1}(X)r_{k+1}(X)\Big)&0
\end{array}
\right]=\\
&=&\left[
\begin{array}{cc}0&X^{kz+km-m+t}\\(-1)^{k+1}X^{kz+km+m-t}p(X)^k&0
\end{array}
\right],
\end{eqnarray*}
so
\begin{equation}\label{eq:1.2rankA^k-divisible}
\mathrm{rank}\, A^k=
\begin{cases}
2n-2k(m+z)&k\le \frac{n-m+t}{m+z},\\
n+m-t-k(m+z)&\frac{n-m+t}{m+z}\le k\le \frac{n+m-t}{m+z},\\
0&k\ge \frac{n+m-t}{m+z}.
\end{cases}
\end{equation}
On the other hand, for $t=m+1$ we use the fact that $r_{k+1}(X)$ has non-zero constant term, which was shown above. Hence $r_{k+1}(X)$ is invertible in $S$ and we obtain
\begin{eqnarray*}
&&\left[
\begin{array}{cc}1&-X\widetilde{r_k}(X)r_{k+1}(X)^{-1}\\0&r_{k+1}(X)^{-1}
\end{array}
\right]\cdot A^k\cdot \left[
\begin{array}{cc}r_{k+1}(X)&0\\-X^mp(X)r_k(X)&1
\end{array}
\right]=\\
&=&\left[
\begin{array}{cc}X^{kz+km}p(X)\Big(r_{k-1}(X)r_{k+1}(X)-r_k(X)^2\Big)&0\\0&X^{kz+km}
\end{array}
\right]=\\
&=&\left[
\begin{array}{cc}X^{kz+km}p(X)^k&0\\0&X^{kz+km}
\end{array},
\right]
\end{eqnarray*}
hence
$$\mathrm{rank}\, A^k=
\begin{cases}
2n-2k(m+z)&k\le \frac{n}{m+z},\\
0&k\ge \frac{n}{m+z},
\end{cases}$$
which is a special case of \eqref{eq:1.2rankA^k-divisible} when $t=m$. If $t=m+1$, we therefore get no additional partitions corresponding to orbits intersecting $\mathcal{N}(B)$, so we do  not consider this case below.

Suppose that $k\le \gamma-1$. If $r_k(X)$ is not divisible by $X$, then \eqref{eq:1.2rankA^k-not_divisible} shows that $\mathrm{rank}\,A^k=2n-k(l+2z)$. On the other hand, if $X^t$ is the highest power of $X$ that divides $r_k(X)$ for some $1\le t\le m$, then $l=2m$ and it is clear that $k\le \frac{k-m+t}{m+z}$, so \eqref{eq:1.2rankA^k-divisible} implies that again we have $\mathrm{rank}\, A^k=2n-2k(m+z)=2n-k(l+2z)$. Similarly, the rank of $A^k$ in \eqref{eq:1.2rankA^k-divisible} is zero if $k\ge \alpha$. It follows that the ranks of $A^k$ in \eqref{eq:1.2rankA^k-not_divisible} and in \eqref{eq:1.2rankA^k-divisible} are different only if $\gamma \le k\le \alpha -1$. Now we consider two cases.

{\bf Case 1.2(a):  $r_k(X)$ is not divisible by $X$ for any $k\in \{\gamma,\ldots,\alpha-1\}$.} 
As noted above, this case in particular includes the case when $l>2m$. We consider various cases with respect to $\alpha$ and $\gamma$:

\begin{enumerate}
\item[(i)] If $\alpha \ge \gamma +2$, then \eqref{eq:1.2rankA^k-not_divisible} implies
$$\mathrm{rank}\, A^k-\mathrm{rank}\, A^{k+1}=\left\{
\begin{array}{ll}l+2z&k\le \gamma -2,\\
l+2z-\delta&k=\gamma-1,\\
m+z&\gamma\le k\le\alpha-2,\\
m+z-\beta&k=\alpha-1,\\
0&k\ge \alpha,
\end{array}
\right.$$
hence the matrix $A$  belongs to  
${\mathcal O}_\bfl$, where  $$\bfl=(\alpha^{m+z-\beta},(\alpha-1)^{\beta},\gamma^{l-m+z-\delta},(\gamma-1)^{\delta})=([n+m]^{m+z},[n-m]^{l-m+z}),$$ 
as in Table~\ref{tab:nn}(P2), for $0\le z\le n-1$, $0<2m\le l<n-z$ and $\alpha \ge \gamma+2$.

\item[(ii)] If $\alpha =\gamma +1$, then we again use \eqref{eq:1.2rankA^k-not_divisible} to get
$$\mathrm{rank}\, A^k-\mathrm{rank}\, A^{k+1}=
\begin{cases}
l+2z&k\le \gamma -2,\\
l+2z-\delta&k=\gamma-1,\\
m+z-\beta&k=\gamma,\\
0&k\ge \gamma +1.
\end{cases}$$
It follows that $A \in {\mathcal O}_\bfl$, where  $$\bfl=((\gamma +1)^{m+z-\beta},\gamma ^{l-m+z+\beta-\delta},(\gamma -1)^{\delta})=([n+m]^{m+z},[n-m]^{l-m+z}),$$ as in Table~\ref{tab:nn}(P2), for $0\le z\le n-1$, $0<2m\le l<n-z$ and $\alpha = \gamma+1$.

\item[(iii)] If $\alpha =\gamma$, then
$$\mathrm{rank}\, A^k-\mathrm{rank}\, A^{k+1}=\left\{
\begin{array}{ll}l+2z&k\le \gamma -2,\\
2n-(\gamma -1)(l+2z)&k=\gamma-1,\\
0&k\ge \gamma,
\end{array}
\right.$$
so $A$ corresponds to the partition Table~\ref{tab:nn}(P1) $$(\gamma ^{2n-(\gamma -1)(l+2z)},(\gamma -1)^{\gamma (l+2z)-2n})=[2n]^{l+2z},$$ where $0\le z\le n-1$, $0<2m\le l<n-z$ and $\alpha= \gamma$.
\end{enumerate}

Conversely, it follows from the above proof that for each partition obtained in (i)-(iii) there indeed exists $A\in \mathcal{N}(B)$ that corresponds to this partition. The only case of existence that might not be obvious is when $l=2m$ and $r_k(X)$ is not divisible by $X$ for any $k\ge 1$. However, in this case we can define $p(X)=q(X)=1$, and then \eqref{eq:r_k-recursion} implies that $r_k(X)$ is a strictly positive constant for each $k\ge 1$, and hence non-zero.

{\bf Case 1.2(b): $l=2m$ and there exists $j\in \{\gamma,\ldots,\alpha-1\}$ such that $r_j(X)$ is divisible by $X$.} Since $r_2(X)=q(X)$ has non-zero constant term, we immediately get $j\ge 3$, and consequently $\alpha \ge 4$. As shown above, we may assume that $r_j(X)=X^t\widetilde{r_j}(X)$ for some $t\le m$ and some polynomial $\widetilde{r_j}(X) \in R$ with non-zero constant term. We have shown above also that $r_{j-1}(X)$ and $r_{j+1}(X)$ are not divisible by $X$. Moreover, since $l=2m$, we have $\frac{n-m}{l-m+z}+2=\frac{n+m+2z}{m+z}\ge \frac{n+m}{m+z}$, so $\gamma +2\ge \alpha$. It follows that $j\in \{\alpha -2,\alpha -1\}$. We consider each of the two options.

\begin{enumerate}
\item[(i)] Assume first that $j=\alpha -2$. Then $\gamma =\alpha -2$ and $\delta =\beta -2z$, so $\beta \ge 2z$ and $\delta <m-z$. (Recall that $\beta <m+ z$.)  Furthermore, from \eqref{eq:1.2rankA^k-divisible} we get
$$\mathrm{rank}\, A^{\alpha -2}=\left\{
\begin{array}{ll}  
2n-2(\alpha-2)(m+z)& m\ge t>\delta, \\
n+m-t-(\alpha-2)(m+z)&   t\le \delta.
\end{array}
\right.$$
Note that the third option in \eqref{eq:1.2rankA^k-divisible} is not possible, since it would imply $t\ge2m-\delta > m+z\ge m$, which contradicts our assumption on $t$.
If $t\le \delta$, then \eqref{eq:1.2rankA^k-not_divisible} and \eqref{eq:1.2rankA^k-divisible} imply 
$$\mathrm{rank}\, A^k-\mathrm{rank}\, A^{k+1}=\left\{
\begin{array}{ll}
2(m+z)&k\le \alpha -4,\\
2m+2z+t-\delta&k=\alpha -3,\\
m+z-t&k=\alpha -2,\\
m+z-\beta&k=\alpha-1,\\0&k \ge \alpha,
\end{array}
\right.$$
so $A$ corresponds to the partition Table~\ref{tab:nn}(P8),
\begin{eqnarray*}
\bfl&=&\left(\alpha^{m-z-\delta},(\alpha -1)^{\delta+2z -t},(\alpha -2)^{m+z+2t-\delta}, (\alpha -3)^{ \delta -t }\right)=\\
&=&\left([n+m-(\alpha -1)t]^{m+z-t},[n-m+(\alpha -1)t]^{m+z+t}\right),
\end{eqnarray*}
where $0\le z\le n-1$, $0<2m<n-z$, $\alpha\ge 5$, $\gamma=\alpha-2$ and $t\le \delta$.
On the other hand, if $m\ge t>\delta$, then 
$$\mathrm{rank}\, A^k-\mathrm{rank}\, A^{k+1}=\left\{
\begin{array}{ll}2(m+z)&k\le \alpha -3,\\
m+z-\delta &k=\alpha -2,\\
m+z-\beta&k=\alpha-1,\\
0&k \ge \alpha,
\end{array}
\right.$$
so $A$ corresponds to the partition Table~\ref{tab:nn}(P8), $(\alpha^{m-z-\delta},(\alpha -1)^{2z},(\alpha -2)^{m+z+\delta})$, which is a special case (for $t=\delta$) of the partition considered in the case when $t\le \delta$.

\item[(ii)] Assume now that $j=\alpha -1$, which is (because of the assumption $\gamma \le j\le \alpha -1$) equal either to $\gamma$ or to $\gamma +1$. Using \eqref{eq:1.2rankA^k-divisible}, we get 
$$\mathrm{rank}\, A^{\alpha -1}=\left\{
\begin{array}{ll}  
2n-2(\alpha-1)(m+z)&t>m-z+\beta, \\
m+z-t-\beta&  t\le \min\{m-z+\beta ,m+z-\beta\},\\
0 & m\ge t>m+z-\beta.
\end{array}
\right.$$
We consider two cases:
\begin{itemize}
\item
Assume first that $j=\gamma+1$. Then, as in (i), we get $\gamma =\alpha -2$ and $\delta =\beta -2z\ge 0$. In particular, we get $t\le m\le m-z+\beta$. 

If $t\le m+z-\beta$, then we get from \eqref{eq:1.2rankA^k-not_divisible} and \eqref{eq:1.2rankA^k-divisible} that
$$\mathrm{rank}\, A^k-\mathrm{rank}\, A^{k+1}=\left\{
\begin{array}{ll}
2(m+z)&k\le \alpha -4,\\
2m+2z-\delta&k=\alpha -3,\\
m+z+t & k=\alpha -2,\\
m+z-t-\beta&k=\alpha-1,\\0&k \ge \alpha,
\end{array}
\right.$$
so $A\in {\mathcal O}_\bfl$, where \begin{eqnarray*}
\bfl&=&\left(\alpha^{m+z-t-\beta},(\alpha -1)^{ \beta+2t},(\alpha -2)^{m+z-t-\delta},(\alpha -3)^{ \delta}\right)=\\
&=&\left([n+m+\gamma t]^{m+z+t},[n-m-\gamma t]^{m+z-t}\right)
\end{eqnarray*} 
as in Table~\ref{tab:nn}(P9) for $0\le z\le n-1$, $0<2m<n-z$, $\gamma=\alpha -2\ge 2$, and $t\le m+z-\beta$.

On the other hand, if $m\ge t>m+z-\beta$, then $$\mathrm{rank}\, A^k-\mathrm{rank}\, A^{k+1}=\left\{
\begin{array}{ll}
2(m+z)&k\le \alpha -4,\\
2m+2z-\delta&k=\alpha -3,\\
2m-\delta & k=\alpha -2,\\
0&k \ge \alpha-1,
\end{array}
\right.$$
so $A$ corresponds to the partition $\left( (\alpha -1)^{ 2m-\delta },(\alpha -2)^{2z},(\alpha -3)^{\delta }\right)$, which is a special case (for $t=m+z-\beta$) of the partition considered above.

\item
Assume now that $j=\gamma$, so $\gamma =\alpha -1$, or, equivalently, $z-m\le \beta <2z$. More precisely, we have $\beta =z-m+\delta$.

If $t\le \min\{m-z+\beta,m+z-\beta\}$, then we get from \eqref{eq:1.2rankA^k-not_divisible} and \eqref{eq:1.2rankA^k-divisible} that

$$\mathrm{rank}\, A^k-\mathrm{rank}\, A^{k+1}=\left\{
\begin{array}{ll}
2(m+z)&k\le \alpha -3,\\
m+3z+t-\beta&k=\alpha -2,\\
m+z-t-\beta&k=\alpha-1,\\0&k \ge \alpha,
\end{array}
\right.$$
so $A\in{\mathcal O}_\bfl$, where
\begin{eqnarray*}
\bfl&=&\left(\alpha^{m+z-t-\beta},(\alpha -1)^{ 2z+2t },(\alpha -2)^{m-z-t+\beta}\right)=\\
&=&\left(\alpha^{m+z-t-\beta},(\alpha -1)^{2t+\beta },(\alpha -1)^{2z-\beta},(\alpha -2)^{m-z-t+\beta}\right)=\\
&=&\left([n+m+(\alpha -2)t]^{m+z+t},[n-m-(\alpha -2)t]^{m+z-t}\right)
\end{eqnarray*}
corresponds to the partition Table~\ref{tab:nn}(P9) for $0\le z\le n-1$, $0<2m<n-z$, $\gamma=\alpha -1\ge 3$ and $t\le \min\{m+z-\beta,m-z+\beta\}$.

If $m\ge t>m+z-\beta$, then 
$$\mathrm{rank}\, A^k-\mathrm{rank}\, A^{k+1}=\left\{
\begin{array}{ll}
2(m+z)&k\le \alpha -3,\\
2n-2(\alpha -2)(m+z)&k=\alpha -2,\\
0&k \ge \alpha-1,
\end{array}
\right.$$
so $A$ corresponds to the partition $$\left((\alpha -1)^{2n-2(\alpha -2)(m+z)},(\alpha -2)^{2(\alpha -1)(m+z)-2n}\right)=[2n]^{2(m+z)},$$ which is of the form Table~\ref{tab:nn}(P1) and whose existence will be shown in Case 2.

Similarly, if $m\ge t>m-z+\beta$, then
$$\mathrm{rank}\, A^k-\mathrm{rank}\, A^{k+1}=\left\{
\begin{array}{ll}
2(m+z)&k\le \alpha -2,\\
2n-2(\alpha -1)(m+z)&k=\alpha -1,\\
0&k \ge \alpha,
\end{array}
\right.$$
so $A$ corresponds to the partition $(\alpha^{2n-2(\alpha -1)(m+z)},(\alpha -1)^{2\alpha (m+z)-2n})$, which is again of the form $[2n]^{2(m+z)}$.
\end{itemize}
\end{enumerate}
 
To finish the proof of Case 1 it remains to show that  when $F=\mathbb{C}$ the partitions obtained in the case 1.2(b) are indeed all possible. First we set the notation. Given polynomials $p(X),q(X)\in R$, we define the sequence of  polynomials $r_k(X)\in R$ by 
$$r_0(X)=0,\quad r_1(X)=1,\quad \mathrm{and}\quad r_{k+1}(X)=q(X)r_k(X)+p(X)r_{k-1}(X)\quad \mathrm{for}\quad k\ge 1.$$
Note that Lemma \ref{A^k} then holds for the polynomials $p(X)$, $q(X)$, and  $r_0(X), r_1(X),\ldots$.

Now let $z\le n-1$ be a non-negative integer, let $m<\frac{n-z}{2}$ be a positive integer, and let $\alpha =\left\lceil \frac{n+m}{m+z}\right\rceil$ and $\beta =\alpha(m+z)-(n+m)$. To show that all partitions obtained in the case~1.2(b) are indeed possible it suffices to show that for $j=\alpha-2\ge 3$ and $1\le t \le\min\{
2z+\beta,m\}
$, and for $j=\alpha-1\ge 3$ and  $1\le t\le \min\{ m+z-\beta,m-z+\beta\}$ there exist polynomials $p(X),q(X)\in R$ with non-zero constant terms such that $X$ does not divide $r_k(X)$ for any $k<\alpha$ satisfying $k\ne j$ and that $X^t$ is the highest power of $X$ that divides $r_j(X)$.  We will show that this happens if  we define $q(X)=1$ and $p(X)=X^t+u$ for a suitable non-zero $u\in \mathbb{C}$. In this case Lemma~\ref{A^k} implies that for $k\ge 1$ we have
$$r_{2k-1}(X)=\sum _{i=1}^k{2k-i-1\choose i-1}u^{i-1}+\sum _{i=1}^k(i-1){2k-i-1\choose i-1}u^{i-2}X^t+\sum_{i\geq 1}v_iX^{t+i}$$
and 
$$r_{2k}(X)=\sum_{i=0}^{k-1}{2k-i-1\choose i}u^i+\sum_{i=0}^{k-1}i{2k-i-1\choose i}u^{i-1}X^t+\sum_{i\geq 1}v'_iX^{t+i}$$
for some $v_i,v_i'\in \mathbb{C}$,
and let $r_0(X)=0$. We define $h_0(u):=0$ and for $k\ge 1$  let
$$h_{2k-1}(u):=\sum_{i=1}^k {2k-i-1\choose i-1}u^{i-1} \text{ and } h_{2k}(u):=\sum_{i=0}^{k-1}{2k-i-1\choose i}u^i.$$
It follows that
\begin{equation}\label{r<->c}
r_k(X)=h_k(u)+h_k'(u)X^t+\, \mathrm{higher}\, \mathrm{terms}
\end{equation}
for all $k\ge 0$ (where $h_k'$ is the derivative of $h_k$). Since $r_k(X)=r_{k-1}(X)+(X^t+u)r_{k-2}(X)$ for $k\ge 2$, we therefore get
$$h_0(u)=0 ,\quad h_1(u)=1, \quad \mathrm{and}\quad h_k(u)=h_{k-1}(u)+uh_{k-2}(u)\,\, \mathrm{ for}\, k\ge 2.$$

Note that we have assumed that $u\ne 0$. Assume now that $u\ne -\frac{1}{4}$ as well. Then we can show, either by solving the above difference equation or by induction, that
$$h_k(u)=\frac{1}{\sqrt{1+4u}}\left(\left(\frac{1+\sqrt{1+4u}}{2}\right)^k-\left(\frac{1-\sqrt{1+4u}}{2}\right)^k\right)$$
for each $k\ge 0$. Here we define the square root as the one that has argument within $[0,\pi)$.  An easy computation shows that
\begin{equation}\label{zeros_of_c}
h_k(u)=0\Leftrightarrow \sqrt{1+4u}=i\, \mathrm{tg}\, \frac{k'\pi}{k}\quad \text{for some }\, k'\in\{1,\ldots,k-1\}\setminus\left\{\frac{k}{2}\right\}.
\end{equation}
 (Note that $k'=0$ is not possible, since we have assumed that $u\ne -\frac{1}{4}$.)
Fix now 
$j\ge 3$ and assume that either $j=\alpha-2$ and $t\le \min\{2z+\beta,m\}$ or $j=\alpha-1$ and $t\le \min\{ m+z-\beta,m-z+\beta\}$ (and hence $\alpha \ge 4$ in both cases). Let $u_j=-\frac{1}{4\cos^2\frac{\pi}{j}}$. Then $\sqrt{1+4u_j}=i\, \mathrm{tg}\, \frac{\pi}{j}$ and $h_j(u_j)=0$ by (\ref{zeros_of_c}).

We now consider $k\ne j$.  We will show that $h_k(u_j)\ne 0$. Assume first that $k<j$. If $k'>\lfloor \frac{k}{2}\rfloor$, then $\mathrm{tg}\, \frac{k'\pi}{k}<0$ and therefore $i\, \mathrm{tg}\, \frac{k'\pi}{k}$ cannot be a square root and hence equal to $\sqrt{1+4u_j}$. On the other hand, if $k'\le \lfloor \frac{k}{2}\rfloor$, then $\frac{k'}{k}\ge \frac{1}{k}>\frac{1}{j}$ and hence $i\, \mathrm{tg}\, \frac{k'\pi}{k}\ne i\, \mathrm{tg}\, \frac{\pi}{j}=\sqrt{1+4u_j}$. On the other hand, if $k>j$, then  the only possibility 
is $j=\alpha -2$ and $k=\alpha -1$, as $k<\alpha$. If $k-1\ge k'>\lfloor  \frac{\alpha-1}{2}\rfloor$, then $\mathrm{tg}\, \frac{k'\pi}{k}<0$ and hence $i\, \mathrm{tg}\, \frac{k'\pi}{k}\ne \sqrt{1+4u_j}$. On the other hand, it is clear that $i\, \mathrm{tg}\, \frac{\pi}{\alpha-1}\ne \sqrt{1+4u_j}$ and that for $2\le k'\le \lfloor \frac{\alpha-1}{2}\rfloor$ we have $\frac{k'}{\alpha-1}>\frac{1}{\alpha-2}$, which again implies $i\, \mathrm{tg}\, \frac{k'\pi}{k}\ne \sqrt{1+4u_j}$.  Using (\ref{zeros_of_c}) we therefore see that $h_k(u_j)=0$ if and only if $k=j$, and (\ref{r<->c}) then implies that $r_k(X)$ is not divisible by $X$  for $k\ne j$ and that $r_j(X)$ is divisible by $X^t$  if we define $p(X)=X^t+u_j$.

It remains to show that $r_j(X)$ is not divisible by $X^{t+1}$, which is by (\ref{r<->c}) equivalent to $h_j'(u_j)\ne 0$. In the neighborhood of $1-\frac{1}{\cos^2\frac{\pi}{j}}$ the square root is a holomorphic function, so we can compute
\begin{align*}
    h_j'(u)=&-\frac{2}{(1+4u)^{\frac{3}{2}}}\left(\left(\frac{1+\sqrt{1+4u}}{2}\right)^j-\left(\frac{1-\sqrt{1+4u}}{2}\right)^j\right)+\\
&+\frac{j}{1+4u}\left(\left(\frac{1+\sqrt{1+4u}}{2}\right)^{j-1}+\left(\frac{1-\sqrt{1+4u}}{2}\right)^{j-1}\right)=\\
=&-\frac{2}{1+4u}h_j(u)+\frac{j}{1+4u}\left(\left(\frac{1+\sqrt{1+4u}}{2}\right)^{j-1}+\left(\frac{1-\sqrt{1+4u}}{2}\right)^{j-1}\right).
\end{align*}
In particular,
$$h_j'(u_j)=\frac{j}{1+4u_j}\left(\left(\frac{1+i\, \mathrm{tg}\, \frac{\pi}{j}}{2}\right)^{j-1}+\left(\frac{1-i\, \mathrm{tg}\, \frac{\pi}{j}}{2}\right)^{j-1}\right)=\frac{j\cos\frac{(j-1)\pi}{j}}{2^{j-2}(1+4u_j)\left(\cos\frac{\pi}{j}\right)^{j-1}}\ne 0$$
for $j\ge 3$, which had to be proved. It follows that all the partitions obtained in case 1.2(b) indeed correspond to nilpotent orbits in $\mathcal{N}(B)$.

\subsection{Case 2: \texorpdfstring{$C_0$}{TEXT} is a zero matrix} 
Since $A$ is non-zero and $z$ is maximal possible, we get that $C_1$ is not nilpotent. We consider two subcases.

\subsubsection{\bf Case 2.1: $C_1$ is invertible}  In this case it is clear that $\mathrm{rank}\, A^k=2n-2k(z+1)$ for $k<\frac{n}{z+1}$ and $A^k=0$ for $k\ge \frac{n}{z+1}$. Let $\alpha=\left\lceil \frac{n}{z+1}\right\rceil$.  Since we have assumed at the beginning of the proof that $A\ne 0$, we have 
$\alpha \ge 2$, and hence 
$$\mathrm{rank}\, A^k-\mathrm{rank}\, A^{k+1}=\left\{
\begin{array}{ll}
2(z+1)&k\le \alpha-2,\\
2n-2(\alpha-1)(z+1) & k=\alpha-1,\\
0&k \geq \alpha,
\end{array}
\right.$$
so $A$ corresponds to the partition Table~\ref{tab:nn}(P1) $$\left(\alpha^{  2n-2(\alpha-1)(z+1)},(\alpha -1)^{2\alpha(z+1)-2n}\right)=[2n]^{2(z+1)}.$$

Conversely, it is clear from the above proof that for each partition $[2n]^{2(z+1)}$ where $0\le z\le n-1$ there exists $A\in \mathcal{N}(B)$ corresponding to this partition. This concludes the proof for the even cases in the first set of partitions stated in the theorem.

\subsubsection{\bf Case 2.2: $C_1$ is not invertible} Since $C_1$ is not nilpotent and the multiplication with a non-zero constant has no influence on the Jordan canonical form of a nilpotent matrix, we may assume that $C_1$ is idempotent, so $C_1^2=C_1$. Moreover, we may conjugate $C$ by an invertible matrix in $\mathcal{C}(B)$ to assume that $C_1=\left[
\begin{array}{cc}1&0\\0&0
\end{array}
\right]$ or, equivalently, $a_1=1$ and $b_1=c_1=d_1=0$. We write matrix  $C$ as $C=\left[
\begin{array}{cc}Xa'(X)&X^2b'(X)\\X^2c'(X)&X^2d'(X)
\end{array}
\right]$ for some polynomials $a'(X),b'(X),c'(X),d'(X)$. As the constant term of $a'(X)$ is 1, it is invertible in $S$. Therefore, by defining first the constant term, then the linear term, and so on, we may find $f(X),r(X)\in S$ such that the equalities
\begin{eqnarray*}
a'(X)\left(f(X)-d'(X)\right)&=&-Xc'(X)b'(X)-Xf(X)d'(X)+Xf(X)^2,\\
a'(X)r(X)&=&X\left(c'(X)+d'(X)r(X)-b'(X)r(X)^2\right)
\end{eqnarray*}
are satisfied in $S$. We furthermore define
$$e(X)=a'(X)+Xb'(X)r(X) \text{ and } q(X)=-Xb'(X)\left(a'(X)-Xf(X)\right)^{-1}$$
in $S$, and the matrices $P=\left[
\begin{array}{cc}1&q(X)\\r(X)&1
\end{array}
\right]$ and $D=\left[
\begin{array}{cc}Xe(X)&0\\0&X^2f(X)
\end{array}
\right]$. Note that $e(X)$ is invertible in $S$, while $q(X)$ and $r(X)$ are divisible by $X$, which implies that matrix  $P$ is invertible in $M_2(S)$. Furthermore, a short calculation shows that $CP=PD$, so we may replace $C$ by a diagonal matrix $D=P^{-1}CP=\left[
\begin{array}{cc}Xe(X)&0\\0&X^mg(X)
\end{array}
\right]$ for some $m\ge 2$ where $e(X),g(X)\in S$ are invertible or $g(X)=0$ and $m=n$. Consequently, $P^{-1}AP=\left[
\begin{array}{cc}
X^{z+1}e(X)&0\\0&X^{z+m}g(X)
\end{array}
\right]$ and it is clear that such a matrix corresponds to the partition Table~\ref{tab:nn}(P5) $$([n]^{z+1},[n]^{z+m})$$ and that all such partitions are possible. Note that here we have $m\ge 2$; the case $m=1$ was considered in the case 2.1. \hfill{\qedsymbol}

\section{Examples and open questions}\label{sec:examples}
We conclude our paper with a few examples that illustrate the methods and results obtained by our work.

\begin{ex}
 Assume that $F=\mathbb{C}$. The list of orbits ${\mathcal O}_{\bfl}$ that have non-empty intersection with the nilpotent centralizer of a nilpotent matrix $B\in {\mathcal O}_{(6,6)}$ were not completely characterized yet, as noted in~\cite[Example~3.7]{Oblak2012}.  Among the undecided partitions $\bfl$ (listed in~\cite[Example~6.25]{oblak2008orbits}) there is only one partition $\bfl=(5,4,3)$ (of type (P3)), which has a non-empty intersection with ${\mathcal N}_B$. The complete list is presented in Table~\ref{tab:66}. 

\begin{table}[htb]
    \centering
    \begin{tabular}{|c|p{13cm}|}
        \hline
       &$\bfl$, such that $\mathcal{O}_{\bfl}$ has a non-empty intersection with ${\mathcal N}(B)$, $B\in \mathcal{O}_{(6,6)}$ \\
        \hline\hline
(P1)& $(12)$, $(6^2)$, $(4^3)$, $(3^4)$, $(3^2,2^3)$, $(2^6)$, $(2^5,1^2)$, $(2^4,1^4)$, $(2^3,1^6)$, $(2^2,1^8)$, $(2,1^{10})$, $(1^{12})$\\
\hline
(P2)&$(7,5)$, $(7,3,2)$, $(7,2^2,1)$, $(7,2,1^3)$, $(7,1^5)$, $(4^2, 2^2)$, $(4^2, 2,1^2)$, $(4^2, 1^4)$, $(4,3^2,2)$, $(4,3,2^2, 1)$, $(4,3,2, 1^3)$, $(4,3,1^5)$, $(3^3,1^3)$, $(3^2,2^2, 1^2)$, $(3^2,2, 1^{4})$, $(3,2^4,1)$, $(3,2^3,1^3)$, $(3,2^2,1^5)$\\
\hline
(P3)&\underline{$(5,4,3)$}\\
\hline
(P4)&$(4,2^4)$\\
\hline
(P5)&$(6,3^2)$, $(6,2^3)$, $(6,2^2,1^2)$, $(6,2,1^4)$, $(6,1^6)$, $(3^2,1^6)$\\
\hline
 \end{tabular}
    \caption{A complete list of the nilpotent orbits intersecting the nilpotent centralizer of $B\in \mathcal{O}_{(6,6)}$. Partitions listed for each type are only the ones that do not appear in lines above it. The underlined partition had not been previously resolved.}
    \label{tab:66}
\end{table}

Moreover, we also provide the complete list of ${\mathcal O}_{\bfl}$ that have non-empty intersection with the nilpotent centralizer of a nilpotent matrix $B\in {\mathcal O}_{(7,7)}$ in Table~\ref{tab:77}. It shows that also partitions of type 
(P6), (P8), and (P9) appear for $n=7$. 

\begin{table}[htb]
    \centering
    \begin{tabular}{|c|p{13cm}|}
        \hline
        &$\bfl$, such that $\mathcal{O}_{\bfl}$ has a non-empty intersection with ${\mathcal N}(B)$, $B\in \mathcal{O}_{(7,7)}$ \\
        \hline\hline
(P1)& $(14)$, $(7^2)$, $(5^2,4)$, $(4^2,3^2)$, $(3^4,2)$, $(3^2,2^4)$, $(2^7)$, $(2^6,1^2)$, $(2^5,1^4)$, $(2^4,1^6)$, $(2^3,1^8)$, $(2^2,1^{10})$, $(2,1^{12})$, $(1^{14})$.\\
\hline
(P2)&$(8,6)$, $(8,3^2)$, $(8,2^3)$, $(8,2^2,1^2)$, $(8,2,1^4)$, $(8,1^6)$, $(5,4,3,2)$, $(5,4,2^2,1)$, $(5,4,2,1^3)$, $(5,4,1^5)$, $(4^2,2^3)$, $(4^2,2^2,1^2)$, $(4^2,2,1^4)$, $(4^2,1^6)$, $(4,3^2,2,1^2)$, $(4,3^2,1^4)$,  $(3^3,2^2,1)$, $(3^3,2,1^3)$, $(3^3,1^5)$, $(3^2,2^3,1^2)$, $(3^2,2^2,1^4)$, $(3^2,2,1^6)$, $(3,2^5,1)$,  $(3,2^4,1^3)$, $(3,2^3,1^5)$, \\
\hline
(P3)&\\
\hline
(P4)&$(6,4^2)$, $(4,3,2^3,1)$, $(4,2^5)$\\
\hline
(P5)&$(7,4,3)$, $(7,3,2^2)$, $(7,2^3,1)$, $(7,2^2,1^3)$, $(7,2,1^5)$, $(7,1^7)$, $(4,3^2,2^2)$, $(4,3,2^2,1^3)$, $(4,3,2,1^5)$, $(4,3, 1^7)$, $(3,2^2,1^7)$ \\
\hline
(P6)&$(4^3,2)$\\
\hline
(P7)&\\
\hline
(P8)&$(5,3^3)$\\
\hline
(P9)&$(3^4,1^2)$\\
\hline
 \end{tabular}
    \caption{A complete list of the nilpotent orbits intersecting the nilpotent centralizer of $B\in \mathcal{O}_{(7,7)}$. Partitions listed for each type are only the ones that do not appear in lines above it.}
    \label{tab:77}
\end{table}

 Note that examples $n=6$ and $n=7$ show different structures of partitions $\lambda$ for which the nilpotent orbit ${\mathcal O}_{\lambda}$ intersects ${\mathcal N}(B)$.  Although partitions of type (P7) do not appear in these two cases, note that for example for $n=22$ ($z=1$, $m=3$, $l=5$, $\alpha=4$, $\beta=5$) the partition $(7^3,6^3,5)$ is of the form (P7) but not (P1)-(P6). This shows that, although types (P1)-(P9) of partitions in Table~\ref{tab:nn} are not complementary, they are all significant to describe ${\mathcal N}(B)$.
\end{ex}

\begin{ex}
Let $n=11$ and consider the nilpotent orbit $\mathcal{O}_{(7,5^3)}$. It can be checked that this orbit is of type (P8) with the corresponding parameters $z=0$, $m=2$, $t=1$, and that it does not belong to any of the types (P1)-(P7) or (P9). We will show that it does not intersect with the nilpotent centralizer of a nilpotent matrix $B\in \mathcal{O}_{(11,11)}$ if $F=\mathbb{Q}$. 

As shown in the proof of Theorem \ref{main}, we may assume that the corresponding matrix in $\mathcal{N}(B)$ is of the form $A=\left[
\begin{array}{cc}0&1\\X^4p(X)&X^2q(X)
\end{array}
\right]$, and let $p_0$ and $q_0$ be the (non-zero) constant terms of $p(X)$ and $q(X)$, respectively. A short calculation shows that
$$A^5=\left[
\begin{array}{cc}X^{10}\Big(2p(X)^2q(X)+p(X)q(X)^3\Big)&X^8\Big(p(X)^2+3p(X)q(X)^2+q(X)^4\Big)\\0&X^{10}\Big(3p(X)^2q(X)+4p(X)q(X)^3+q(x)^5\Big)
\end{array}
\right].$$
By the assumption the matrix $A$ corresponds to the partition $(7,5^3)$, which implies that $\mathrm{rank}\, A^5=2$, and consequently the constant term of $p(X)^2+3p(X)q(X)^2+q(X)^4$ has to be zero. However, the equation $p_0^2+3p_0q_0^2+q_0^4=0$ has no non-zero solutions in $\mathbb{Q}$, which shows that the nilpotent orbit $\mathcal{O}_{(7,5^3)}$ does not intersect the nilpotent centralizer of $B$ if $F=\mathbb{Q}$.
\end{ex}

\begin{ex}
Let $n=9$ and consider the nilpotent orbit $\mathcal{O}_{(5^3,3)}$, which is of type (P9) with parameters $z=0$, $m=2$, $t=1$, and it does not belong to any of the types (P1)-(P8). Let $A=\left[
\begin{array}{cc}0&1\\X^4p(X)&X^2q(X)
\end{array}
\right]$ be the corresponding matrix in $\mathcal{N}(B)$. Then
$$A^5=
\left[
\begin{array}{cc}0&X^8\Big(p(X)^2+3p(X)q^2(X)+q(X)^4\Big)\\0&0
\end{array}
\right],$$
which has to be the zero matrix, as $A$ corresponds to the partition $(5^3,3)$. As in the previous example we get the equation $p_0^2+3p_0q_0^2+q_0^4=0$, which has no non-zero solutions in $\mathbb{Q}$. It follows that the nilpotent orbit $\mathcal{O}_{(5^3,3)}$ of type (P9) does not intersect the nilpotent centralizer of $B\in \mathcal{O}_{(9,9)}$ if $F=\mathbb{Q}$.
\end{ex}

Although we were able to characterize nilpotent orbits ${\mathcal O}_{\bfl}$ that have non-empty intersection with the nilpotent centralizer of a nilpotent matrix $B\in {\mathcal O}_{(n,n)}$, the partitions $\bfl$ in Table~\ref{tab:nn} are presented in a rather scattered way. 

\begin{que}
 Is there a combinatorial description of partitions obtained in Table~\ref{tab:nn}?
\end{que}

Moreover, there are plenty of possible generalizations to Theorem~\ref{main} that would be interesting to resolve.

\begin{que}
 Do any of the above methods apply to the case when:
  \begin{enumerate}
      \item $B\in {\mathcal O}_{(\mu_1,\mu_2)}$,  $\mu_1\ne \mu_2$, i.e., in the case when the matrix $B$ still has just two Jordan blocks, but different in size,
      \item $B\in {\mathcal O}_{(n^\ell)}$,  $\ell \geq 3$, i.e., in the case when the matrix $B$ has at least three Jordan blocks, all of the same size,
      \item we consider some special nilpotent orbits over other simple Lie algebras?
  \end{enumerate}
  Or, alternatively, could we use the results from Theorem~\ref{main} to further extend the partial answers to open questions mentioned in this paper?
\end{que}

\bibliographystyle{plain}
\bibliography{bibliography}

\end{document}